\documentclass{amsart}
\usepackage{amsmath, amsthm, amscd, amsfonts, amssymb, graphicx, color, cite}
\usepackage{amsaddr}

\newtheorem{theorem}{Theorem}
\newtheorem{lemma}[theorem]{Lemma}

\newtheorem{proposition}[theorem]{Proposition}

\begin{document}

\title[$b$-chromatic number of subdivision-vertex neighbourhood coronas]{Determining the $b$-chromatic number of subdivision-vertex neighbourhood coronas}

\author{%
  Ra\'ul M. Falc\'on$^{1,*}$,
  M. Venkatachalam $^2$
  and
 S. Julie Margaret$^2$
}

\address{$^1$ Dept. Applied Mathematics I, Universidad de Sevilla, Spain.\\
$^2$ Dept. Mathematics, Kongunadu Arts and Science College, India.}
\email{$^*$ rafalgan@us.es}

\begin{abstract}
Let $G$ and $H$ be two graphs, each one of them being a path, a cycle or a star. In this paper, we determine  the $b$-chromatic number of every  subdivision-vertex neighbourhood corona $G\boxdot H$ or $G\boxdot K_n$, where $K_n$ is the complete graph of order $n$. It is also established for those graphs $K_n\boxdot G$ having $m$-degree not greater than $n+2$. All the proofs are accompanied by illustrative examples.


\end{abstract}

\maketitle

\section{Introduction}

In 1999, Irving and Manlove \cite{Irving1999} introduced the {\em $b$-chromatic coloring} of a graph $G=(V(G),E(G))$ as a proper $k$-coloring $c:V(G)\rightarrow \{0,\ldots,k-1\}$ with a {\em $b$-vertex} for each color $i$. That is, a vertex $v\in V(G)$ such that $c(v)=i$ and, for each color $j\neq i$, there exists a vertex $w\in N_G(v)$ satisfying that $c(w)=j$. Here, $N_G(v)$ denotes the neighborhood of the vertex $v$. The {\em $b$-chromatic number} $\varphi(G)$ is the maximum positive integer $k$ for which a $b$-chromatic coloring of $G$ with $k$ colors exists. Any such a $b$-chromatic coloring is said to be {\em optimal}. Irving and Manlove proved that the problem of determining the $b$-chromatic number of a graph is NP-hard in general, and polynomial-time solvable for trees. It has been dealt with by a wide amount of graph theorists (see \cite{Jakovac2018} for a survey). Of particular interest for the aim of this paper, it is remarkable the study of the $b$-chromatic number of distinct graph products as the Cartesian product \cite{Kouider2002, Kouider2007, Maffray2013, Balakrishnan2014, Balakrishnan2016, Balakrishnan2017, Javadi2012, Kerdjoudj2017, Guo2018, Afrose2020}, the direct product \cite{Jakovac2012, Koch2015}, the strong product \cite{Jakovac2012, Raj2017}, the lexicographic product \cite{Jakovac2012, Raj2017, Linhares2017}, the corona product \cite{Vivin2012} or the subdivision edge and vertex corona \cite{Pathinathan2014}.

This paper delves into this topic for the subdivision-vertex neighborhood corona (from here on, SVN corona) of paths, cycles, stars and complete graphs. Recall here that the {\em subdivision graph} $S(G)$ of a graph $G$ arises from inserting a new vertex into every edge of $G$. In 2013, Liu and Lu \cite{Liu2013} introduced the {\em SVN corona} $G \boxdot H$ of two graphs $G$ and $H$, with $V(G)=\{u_0,\ldots,u_{n-1}\}$, as the graph arising from adding $n$ vertex-disjoint copies of $H$ to $S(G)$, so that every vertex in $N_{S(G)}(u_i)$ is joined to every vertex in the $(i+1)^{\mathrm{th}}$ copy of $H$, for all $i<n$. 

The paper is organized as follows. In Section~\ref{sec:preliminaries}, we describe some preliminary concepts and results on Graph Theory that are used throughout the manuscript. Then, Sections \ref{sec:path}--\ref{sec:complete} deal separately with the $b$-chromatic number of SVN coronas of paths, cycles, stars and complete graphs.

\section{Preliminaries} \label{sec:preliminaries}

All the graphs throughout this paper are finite and simple. This section deals with some notations and preliminary results on graph theory that are used throughout the paper. 

Let $G=(V(G),E(G))$ be a graph. The neighborhood and degree of a vertex $v\in V(G)$ are respectively denoted by $N_G(v)$ and $d_G(v)$. If there is no risk of confusion, then we use the respective notations $N(v)$ and $d(v)$. In addition, $\Delta(G)$ denotes the maximum vertex degree of the graph $G$. Further, the path, cycle and star of order $n>2$ are respectively denoted by $P_n$, $C_n$, and $S_{n-1}$. The complete graph of order $n$ is denoted by $K_n$. 

A {\em proper $k$-coloring} of a graph $G$ is any map $c:V(G)\rightarrow \{0,\ldots,k-1\}$ assigning $k$ {\em colors} to the set of vertices $V(G)$ so that no two adjacent vertices share the same color. The {\em chromatic number} $\chi(G)$ is the minimum positive integer $k$ for which a proper $k$-coloring of $G$ exists. An example of proper $k$-coloring is the $b$-chromatic coloring with $k$ colors that has been described in the introductory section.

\begin{lemma}\label{lemma0}\cite{Irving1999} Let $G$ be a graph. Then, $\chi(G) \leq \varphi(G) \leq \Delta (G) + 1$.
\end{lemma}

\begin{proposition}\label{proposition_Kouider}\cite{Kouider2002} Let $n$ be a positive integer. Then,
\begin{itemize}
    \item $\varphi(P_n)=\begin{cases}
    \begin{array}{ll}
    2, & \text{ if } n\in\{3,4\},\\
    3, & \text{ if } n>4.
    \end{array}
    \end{cases}$
    \item $\varphi(C_n)=\begin{cases}
    \begin{array}{ll}
    2, & \text{ if } n=4,\\
    3, & \text{ if either } n=3 \text{ or } n>4.
    \end{array}
    \end{cases}$
    \item $\varphi(S_n)=2$, for all $n>2$.
    \item $\varphi(K_n)=n$.
\end{itemize}
\end{proposition}

\vspace{0.2cm}

In order to study the $b$-chromatic number of any graph, Irving and Manlove defined the {\em $m$-degree} of a graph $G$ of order $n$ as 
\[m(G):=\left|\{i\in\{1,\ldots,n\}\colon\,d(v_{i-1})\geq i-1\}\right|,\]
where $V(G)=\{v_0,\ldots,v_{n-1}\}$ is such that $d(v_0)\geq \ldots\geq d(v_{n-1})$.

\begin{lemma}\label{lemma1}\cite{Irving1999} Let $G$ be a graph. Then, $\varphi(G) \leq m(G)$.
\end{lemma}

\vspace{0.2cm}

We finish this preliminary section by introducing some notation concerning any SVN corona $G\boxdot H$. Here, we assume that $V(G)=\{u_0,\ldots,u_{|V(G)|-1}\}$ and  $V(H)=\{v_0,\ldots,v_{|V(H)|-1}\}$. In addition, let $I(G)$ denote the set of vertices that are inserted into the edges of the graph $G$ to get the subdivision graph $S(G)$. Then, we use the following notation throughout the paper.
\begin{itemize}
\item $s_{i,j}$ denotes the vertex in $I(G)$ that is inserted in an edge $u_iu_j\in E(G)$. Depending on convenience, we may also denote this vertex by $s_{j,i}$.

\item $v_{i,j}$ denotes the copy of each vertex $v_j\in V(H)$ in the $(i+1)^{\mathrm{th}}$ copy of the graph $H$.
\end{itemize}
In the constructive proofs of the paper, all the indices of the just described vertices are considered to be modulo either $|V(G)|$ or $|V(H)|$ (depending on the case). Furthermore, concerning the graphical representation of the SVN corona $G\boxdot H$, edges between $S(G)$ and one of the copies of $H$ are drawn with dashed lines. Figure \ref{Fig_Corona} illustrates these notations.

\begin{figure}[htbp]
\centering
\includegraphics[scale=0.15]{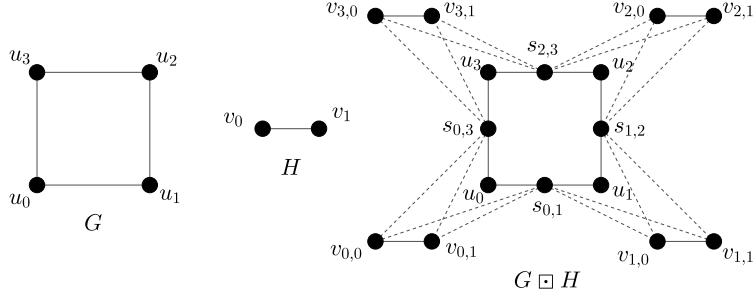}
\caption{SVN corona product.}\label{Fig_Corona}
\end{figure}

In particular,
{\small \begin{equation}\label{eq_d}
d_{G\boxdot H}(v)=\begin{cases}
\begin{array}{ll}
d_G(v),& \text{ if } v\in V(G),\\
2\cdot |V(H)| +2, & \text{ if } v\in I(G),\\
d_G(u_i) + d_H(v_j), & \text{ if } v=v_{i,j},\text{ for some } \begin{cases}
0\leq i<|V(G)|,\\
0\leq j<|V(H)|.
\end{cases}
\end{array}
\end{cases}
\end{equation}}

\begin{proposition}\label{prop_2t4} The following statements hold.
\begin{enumerate}
\item[(a)] $\Delta(G\boxdot H)=\max\{2\cdot |V(H)|+2,\,\Delta(G)+\Delta(H)\}$.

\vspace{0.2cm}

\item[(b)] If $\Delta(G)\leq |V(H)|+3$, then: 

\vspace{0.2cm}

\begin{enumerate}
    \item[(b.1)] $\varphi(G\boxdot H)\leq 2\cdot|V(H)|+3$.

    \vspace{0.2cm}

    \item[(b.2)] If $\Delta(G)+\Delta(H)+1\leq |I(G)|\leq 2\cdot|V(H)|+2$, then $\varphi(G\boxdot H)\leq |I(G)|$.
\end{enumerate}
\end{enumerate}
\end{proposition}

\begin{proof} The first statement follows readily from (\ref{eq_d}). Thus, in what follows, we assume that $\Delta(G)\leq |V(H)|+3$. Since $\Delta(H)\leq |V(H)|-1$, we have that $\Delta(G\boxdot H)=2\cdot|V(H)|+2$, whenever $\Delta(G)\leq |V(H)|+3$. Hence, (b.1) follows from Lemma \ref{lemma0}. Furthermore, there are $|I(G)|$ vertices in $G\boxdot H$ having maximum degree $2t+2$. From (\ref{eq_d}), the highest vertex degree in $G\boxdot H$ being less than this maximum is $\Delta(G)+\Delta(H)$. Thus, the assumptions of (b.2) imply that $m(G\boxdot H)=|I(G)|$. Hence, the last statement follows from Lemma \ref{lemma1}.
\end{proof}

Finally, in order to make easier the identification of $b$-vertices in the constructive proofs described in the next four sections, we define a {\em $b$-rainbow set} of a graph to be any set formed by exactly one $b$-vertex of each one of the colors associated to an optimal $b$-chromatic coloring of the graph under consideration. In the illustrative figures of this manuscript, vertices of a particular $b$-rainbow set are represented by crosses $\mathbf{\times}$; other $b$-vertices are represented by triangles $\blacktriangle$ and the remaining vertices are represented by circles $\bullet$.

\section{SVN corona of paths}\label{sec:path}

From here on, let $\mathcal{G}$ denote the set of paths, cycles, stars and complete graphs of any order. In this section, we determine the $b$-chromatic number of the SVN corona $P_n\boxdot G$ of a path $P_n=\langle\,u_0,\ldots,u_{n-1}\,\rangle$, with $n>2$, and a graph $G\in\mathcal{G}$. As a preliminary result, Proposition \ref{prop_2t4} enables us to study this number for the SVN corona $P_n\boxdot H$, for any arbitrary graph $H$ such that $n\geq \Delta(H)+4$.

\begin{proposition}\label{proposition_2t4} The following statements hold.
\begin{enumerate}
\item[(a)] $\varphi(P_n\boxdot H)\leq n-1$, whenever $\Delta(H)+4\leq n\leq 2\cdot |V(H)|+3$.

\vspace{0.2cm}

\item[(b)] $\varphi(P_n\boxdot H)=2\cdot |V(H)|+3$, whenever $n>2\cdot |V(H)|+3$.
\end{enumerate}
\end{proposition}

\begin{proof} The first statement follows readily from (b.2) in Proposition \ref{prop_2t4} once it is observed that $\Delta(P_n)=2$ and $|I(P_n)|=n-1$. Further, since $\Delta(P_n)=2<|V(H)|+3$, we have from (b.1) in Proposition \ref{prop_2t4} that $\varphi(P_n\boxdot H)\leq 2\cdot |V(H)|+3$. In order to prove that this upper bound is reached, it is enough to define the $b$-chromatic coloring $c$ of the graph $P_n\boxdot H$ such that, for each triple of non-negative integers $i<n$, $j<n-1$ and $k<|V(H)|$, we have that $c(u_i)=(i+1)\,\mathrm{mod}\, (2\cdot |V(H)|+3)$, $c(s_{j,j+1})=j\,\mathrm{mod}\,(2\cdot |V(H)|+3)$ and $c(v_{i,k})=(i+2k+3)\,\mathrm{mod}\,(2\cdot |V(H)|+3)$. A $b$-rainbow set is formed by the vertices $s_{0,1},\ldots,s_{2\cdot |V(H)|+2,2\cdot |V(H)|+3}$. 
\end{proof}

\vspace{0.2cm}

Figure \ref{Fig_P10P3} illustrates the $b$-chromatic coloring described in the previous theorem for $(n,t)=(10,3)$.

\begin{figure}[ht]
\begin{center}
\includegraphics[scale=0.6]{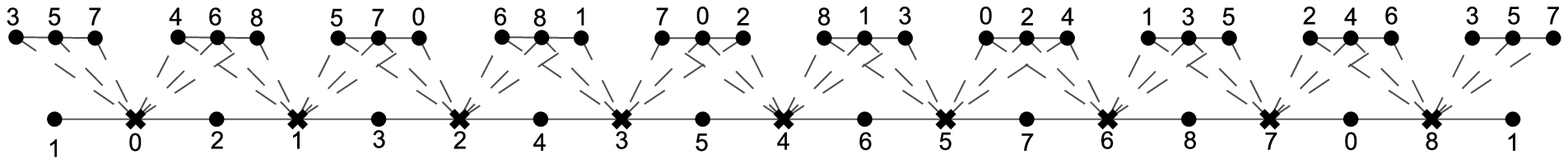}
\caption{Optimal $b$-chromatic coloring of $P_{10}\boxdot P_3$.}
\label{Fig_P10P3}
\end{center}
\end{figure}

Now, we focus separately on each one of the mentioned graphs $P_n\boxdot G$, with $G\in\mathcal{G}$. In all the proofs, we define an appropriate  $b$-chromatic coloring $c$ of the graph $P_n\boxdot G$ such that
\begin{equation}\label{eq_Pn}
\begin{cases}
c(u_i)=(i+1)\,\mathrm{mod}\,\min\{m(P_n\boxdot P_t),\,n\},\\
c(s_{j,j+1})=j\,\mathrm{mod}\,\min\{m(P_n\boxdot P_t),\,n\}.
\end{cases}
\end{equation}
for every pair of non-negative integers $i<n$ and $j<n-1$. We start by determining the $b$-chromatic number for the SVN corona of two paths.

\begin{theorem}\label{theorem_path} Let $n>2$ and $t>2$ be two positive integers. Then,
\[\varphi(P_n\boxdot P_t)=\begin{cases}
\begin{array}{lll}
4,& \text{ if } n=3 \text{ and } t\in\{3,4\},\\
5,& \text{ if } \begin{cases}
n=3 \text{ and } t\geq 5,\\
n\in\{4,5\},
\end{cases}\\
n-1,& \text{ if } 6\leq n\leq 2t+3,\\
2t+3,& \text{ otherwise}.
\end{array}
\end{cases}\]
\end{theorem}

\begin{proof} The case $n>2t+3$ follows from Proposition \ref{proposition_2t4}. So, we assume from now on that $n\leq 2t+3$. From Lemma \ref{lemma1} and (\ref{eq_d}), we have that all the described values are upper bounds of the $b$-chromatic number under consideration. In order to see that they are reached, we define an appropriate $b$-chromatic coloring $c$ of the graph $P_n\boxdot P_t$ satisfying (\ref{eq_Pn}). For each pair of non-negative integers $i<n$ and $j<t$, the following two cases arise. Here, we assume that $P_t=\langle\,v_0,\ldots,v_{t-1}\,\rangle$. 

If $n\leq 6$, then
    \[c(v_{i,j})=\begin{cases}
    \begin{array}{ll}
        (i+(j\,\mathrm{mod}\,2) + 1)\,\mathrm{mod}\, 4, & \text{ if } n=3 \text{ and } t\in\{3,4\},\\
        (i+(j\,\mathrm{mod}\,3)+1)\,\mathrm{mod}\,5, & \text{ otherwise}.
    \end{array}
    \end{cases}\]
A $b$-rainbow set is formed by the vertices $s_{0,1},\ldots,s_{n-2,n-1}$, together with the vertex $v_{1,1}$, if $n<6$; the vertex $v_{1,2}$, if $n\in\{3,4\}$; and the vertex $v_{1,3}$, if $n=3$ and $t\geq 5$. (Figure \ref{Fig_P4P4} illustrates the case $n=t=4$.)

\begin{figure}[ht]
\begin{center}
\includegraphics[scale=0.15]{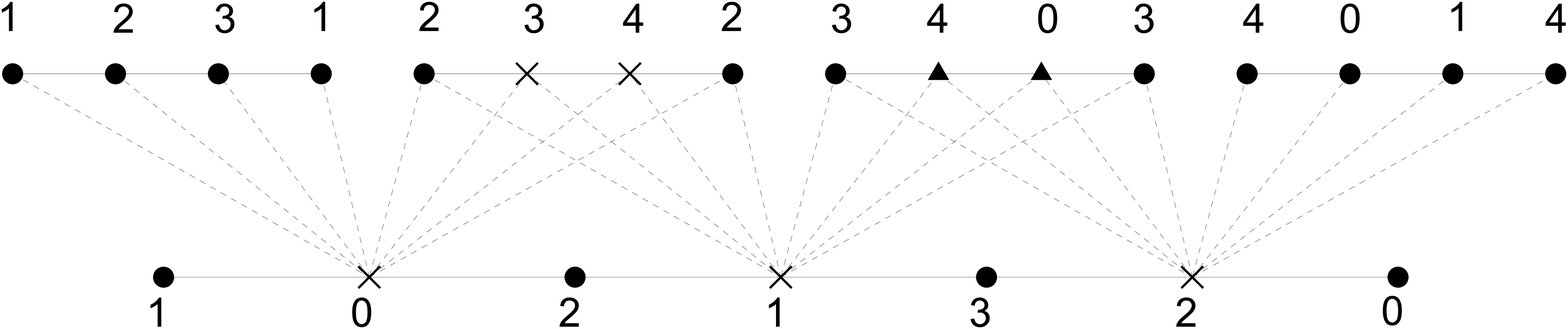}
\caption{Optimal $b$-chromatic coloring of $P_4\boxdot P_4$.}
\label{Fig_P4P4}
\end{center}
\end{figure}
    
Further, if $7\leq n\leq 2t+3$, then $c(v_{i,j})$ is
\[\begin{cases}
    \begin{array}{ll}
    (i+2j+3)\,\mathrm{mod}\,(n-1), & \text{ if }  n \text{ is even and }j<\frac {n-4}2,\\
        \left(i+2j+(4-2(i\,\mathrm{mod}\,2))\right)\,\mathrm{mod}(n-1), & \text{ if } \begin{cases}n,  i+1 \text{ are odd and } j<\left\lfloor\,\frac {n-4}2\,\right\rfloor,\\
        n, i \text{ are odd and } j<\left\lceil\,\frac {n-4}2\,\right\rceil,
        \end{cases}\\
        (i+2)\,\mathrm{mod}\,6,&\text{ if }
        (j,n)=(1,7) \text{ and } i \text{ is even},\\
        c(v_{i,j-2}),&\text{ otherwise}.
    \end{array}
    \end{cases}\]
A $b$-rainbow set is formed by the vertices $s_{0,1},\ldots,s_{n-2,n-1}$. (Figure \ref{Fig_P7P4} illustrates the case $(n,t)\in\{(7,4),\,(8,4)\}$.)
\end{proof}

\begin{figure}[ht]
\begin{center}
\begin{tabular}{c}
\includegraphics[scale=0.085]{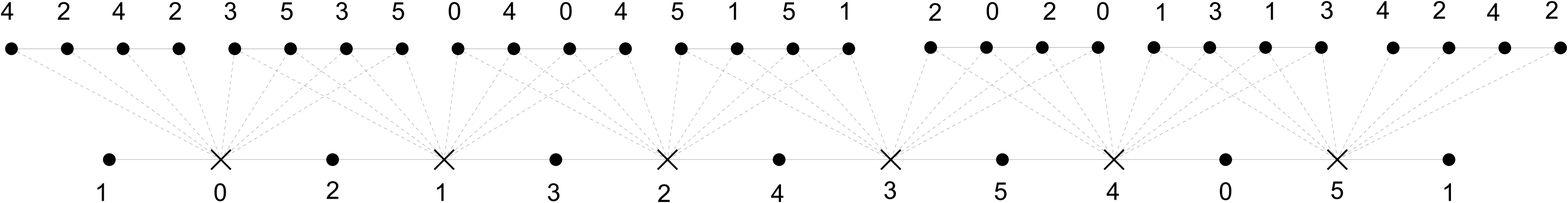}\\
\includegraphics[scale=0.075]{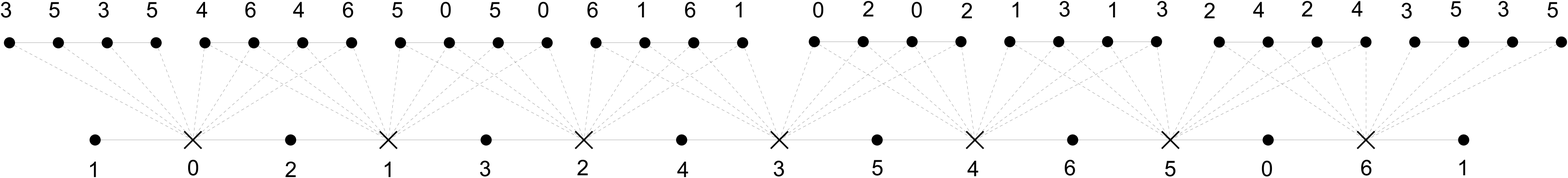}
\end{tabular}
\caption{Optimal $b$-chromatic colorings of $P_{7}\boxdot P_4$ and $P_{8}\boxdot P_4$.}
\label{Fig_P7P4}
\end{center}
\end{figure}

The next graph to study is the SVN corona of a path and a cycle.

\begin{theorem}\label{theorem_PnCt} Let $n>2$ and $t>2$ be two positive integers. Then,
\[\varphi(P_n\boxdot C_t)=\begin{cases}
\begin{array}{ll}
4, & \text{ if } (n,t)=(3,4),\\
5, & \text{ if } \begin{cases}
n=3 \text{ and } t\neq 4,\\
n\in\{4,5\},
\end{cases}\\
n-1, & \text{ if } 6\leq n\leq 2t+3,\\
2t+3, & \text{ otherwise}.
\end{array}
\end{cases}\]
\end{theorem}

\begin{proof} The case $n>2t+3$ follows from Proposition \ref{proposition_2t4}. So, we assume from now on that $n\leq 2t+3$. From Lemma \ref{lemma1} and (\ref{eq_d}), we have that
\begin{equation}\label{eq_PC}
 \varphi(P_n\boxdot C_t)\leq m(P_n\boxdot C_t)=\begin{cases}
\begin{array}{ll}
5, & \text{ if } n\in\{3,4,5\},\\
n-1, & \text{ otherwise}.
\end{array}
\end{cases}   
\end{equation}
This upper bound is not reached for $(n,t)=(3,4)$, because every  $b$-chromatic coloring of the SVN corona $P_3\boxdot C_4$ with five colors would imply the existence of a $b$-chromatic coloring of the cycle $C_4$ with three colors. It is not possible, because $\varphi(C_4)=2$ (see Proposition \ref{proposition_Kouider}). Hence, $\varphi(P_3\boxdot C_4)\leq 4$. Figure \ref{Fig_P3C4} shows that this upper bound is indeed reached.

\begin{figure}[ht]
\begin{center}
\includegraphics[scale=0.15]{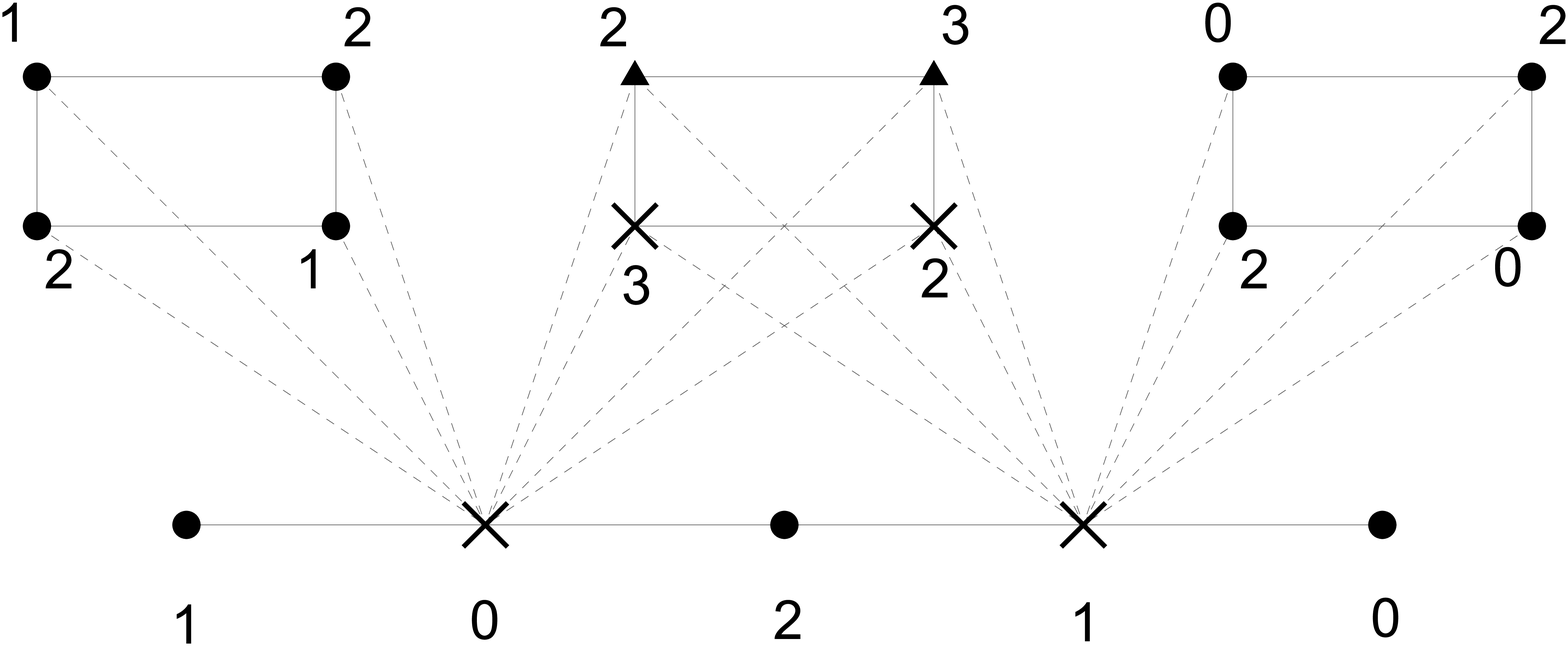}
\caption{Optimal $b$-chromatic coloring of $P_3\boxdot C_4$.}
\label{Fig_P3C4}
\end{center}
\end{figure}

To prove that the upper bound in $(\ref{eq_PC})$ is reached for all $(n,t)\neq (3,4)$, we define an appropriate $b$-chromatic coloring $c$ of $P_n\boxdot C_t$ satisfying (\ref{eq_Pn}). Two cases arise for each pair of non-negative integers $i<n$ and $j<t$. (Here, we assume that $C_t=\langle\,v_0,\ldots,v_{t-1},v_0\,\rangle$.) 

First, if $n\leq 6$ and $(n,t)\neq (3,4)$, then
    \[c(v_{i,j})=\begin{cases}
    \begin{array}{ll}
    2, & \text{ if } (n,i,j)=(3,1,2),\\
    (i+2 + (j\,\mathrm{mod}\,2))\,\mathrm{mod}\, 5,& \text{ if } j\neq t-1 \text{ and } (n,i,j)\neq (3,1,2),\\
    i+1, &\text{ if } j=t-1.
    \end{array}
    \end{cases}\]
A $b$-rainbow set is formed by the vertices $s_{0,1},\ldots,s_{n-2,n-1}$, together with the vertex $v_{2,0}$, if $n\in\{4,5\}$; the vertex $v_{1,0}$, if $n\in\{3,4\}$; and the vertices $v_{1,1}$ and $v_{1,t-1}$, if $n=3$. (Figure \ref{Fig_P4C4} illustrates the case $n=t=4$.)

Second, if $7\leq n\leq 2t+3$, then we define $c(v_{i,j})$ as in the proof of Theorem \ref{theorem_path}, except for
    \[c(v_{i,j})=\begin{cases}
    \begin{array}{ll}
    i+1,& \text{ if } n=8, t \text{ is odd and } j=t-1,\\
    i+3,& \text{ if } n\in\{7,9\}, t \text{ is odd and } j=t-1.
    \end{array}
    \end{cases}\]
A $b$-rainbow set is formed by the vertices $s_{0,1},\ldots,s_{n-2,n-1}$.
\end{proof}

\begin{figure}[ht]
\begin{center}
\includegraphics[scale=0.15]{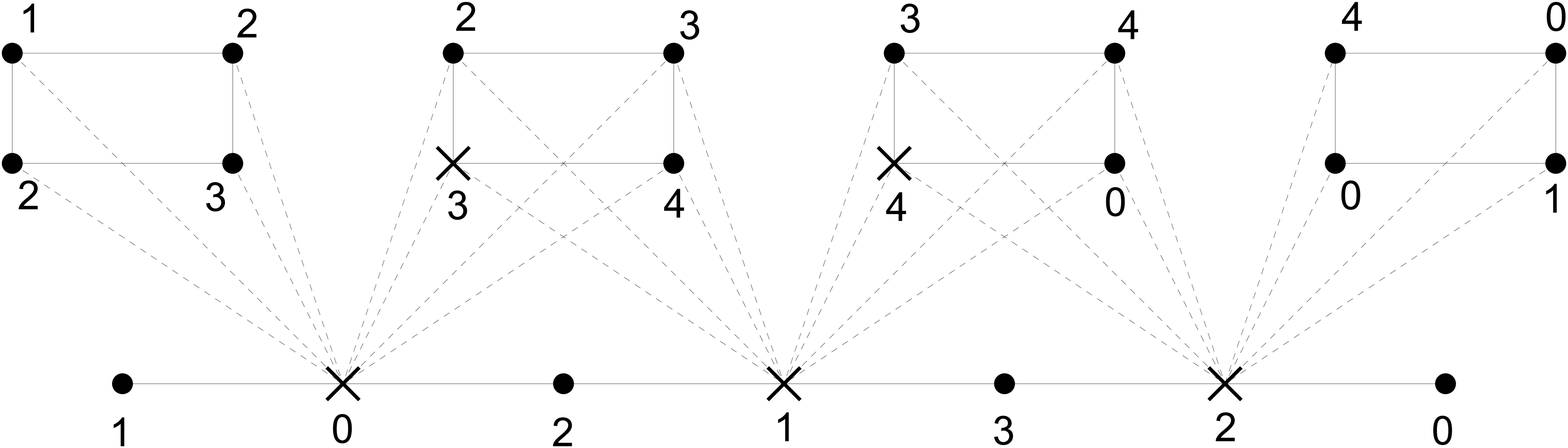}
\caption{Optimal $b$-chromatic coloring of $P_4\boxdot C_4$.}
\label{Fig_P4C4}
\end{center}
\end{figure}

Now, we study the SVN corona of a path and a star. (Recall here that $S_t$ is the star with $t$ leaves, so its order is $t+1$. It is important for the application of Proposition \ref{proposition_2t4}.)

\begin{theorem}\label{theorem_PnSt}
Let $n>2$ and $t>2$ be two positive integers. Then,
\[\varphi(P_n \boxdot S_t) = \begin{cases}
\begin{array}{lll}
2n-1, & \text{ if } n\leq \frac {t+3}2,\\
t+2, & \text{ if }n=\lceil\,\frac {t+4}2\,\rceil,\\
t+3, & \text{ if } \lceil\,\frac {t+4}2\,\rceil < n \leq  t+4,\\
n-1,&\text{ if } t+4<n\leq 2t+5,\\
2t+5, & \text{ otherwise}.
\end{array}
\end{cases}\]
\end{theorem}

\begin{proof} The case $n>2t+5$ follows from Proposition \ref{proposition_2t4}.  So, we may assume that $n\leq 2t+5$. From Lemma \ref{lemma1} and (\ref{eq_d}), all the described values are upper bounds of the $b$-chromatic number. To see that they are reached, we define an appropriate $b$-chromatic coloring $c$ of the graph $P_n\boxdot S_t$ satisfying (\ref{eq_Pn}). For each pair of non-negative integers $i<n$ and $j\leq t$, the following two cases arise. Here, we assume that $V(S_t)=\{v_0,\ldots,v_t\}$, where $v_0$ is the center of the star.  

If $n\leq t+4$, then let $\alpha_{n,t}=m(P_n\boxdot S_t)-n+1$. Then, we define $c(v_{i,j})$ as
 \[\begin{cases}
    \begin{array}{ll}
    n-1+((i+j)\,\mathrm{mod}\,\alpha_{n,t}), & \text{ if } j<\alpha_{n,t},\\
    (j-\alpha_{n,t}+1)\,\mathrm{mod}\, (n-1),& \text{ if } \alpha_{n,t}\leq j<\alpha_{n,t} + n-2 \text{ and } i=0,\\
    (j-\alpha_{n,t})\,\mathrm{mod}\, (n-1),& \text{ if } \alpha_{n,t}\leq j<\alpha_{n,t} + n-2 \text{ and } i=n-1,\\
       (i+j-\alpha_{n,t}+1)\,\mathrm{mod}\, (n-1),& \text{ if } \alpha_{n,t}\leq j<\alpha_{n,t} + n-3 \text{ and } i\not\in\{0,n-1\},\\
    c(v_{i,j-1}),&\text{ otherwise}.
    \end{array}
    \end{cases}\]
 A $b$-rainbow set is formed by the vertices $s_{0,1},\ldots,s_{n-2,n-1}$, together with either the vertices $v_{0,0},\ldots,v_{n-1,0}$, if $n\leq \frac {t+3}2$; or the vertices $v_{0,0},\ldots,v_{t-n+2,0}$, if $n=\lceil\frac{t+4}2\rceil$; or the vertices $v_{0,1},\ldots,v_{t-n+4,0}$, if $\lceil\frac{t+4}2\rceil< n\leq 2t+5$. (Figure \ref{Fig_P3S4} illustrates the case $(n,t)\in\{(3,4),\,(4,4),\,(6,4)\}$.)
     
Further, if $t+4<n\leq 2t+5$, then
 \[c(v_{i,j})=\begin{cases}
    \begin{array}{ll}
    (i+3)\,\mathrm{mod}\,(n-1),& \text{ if } j=0,\\
    (i-j-2)\,\mathrm{mod}\,(n-1), & \text{ if } i \text{ is even and } 0<j\leq\min\{n-6 , t\},\\
    (i+j+3)\,\mathrm{mod}\,(n-1), & \text{ if } i \text{ is odd and } 0<j\leq\min\{n-6 , t\},\\
    c(v_{i,j-1}), & \text{ otherwise}.
    \end{array}
    \end{cases}\]
    A $b$-rainbow set is formed by the vertices $s_{0,1},\ldots,s_{n-2,n-1}$. (Figure \ref{Fig_P9S4} illustrates the graph $P_9\boxdot S_4$.)
\end{proof}

\begin{figure}[ht]
\begin{center}
\begin{tabular}{c}
\includegraphics[scale=0.12]{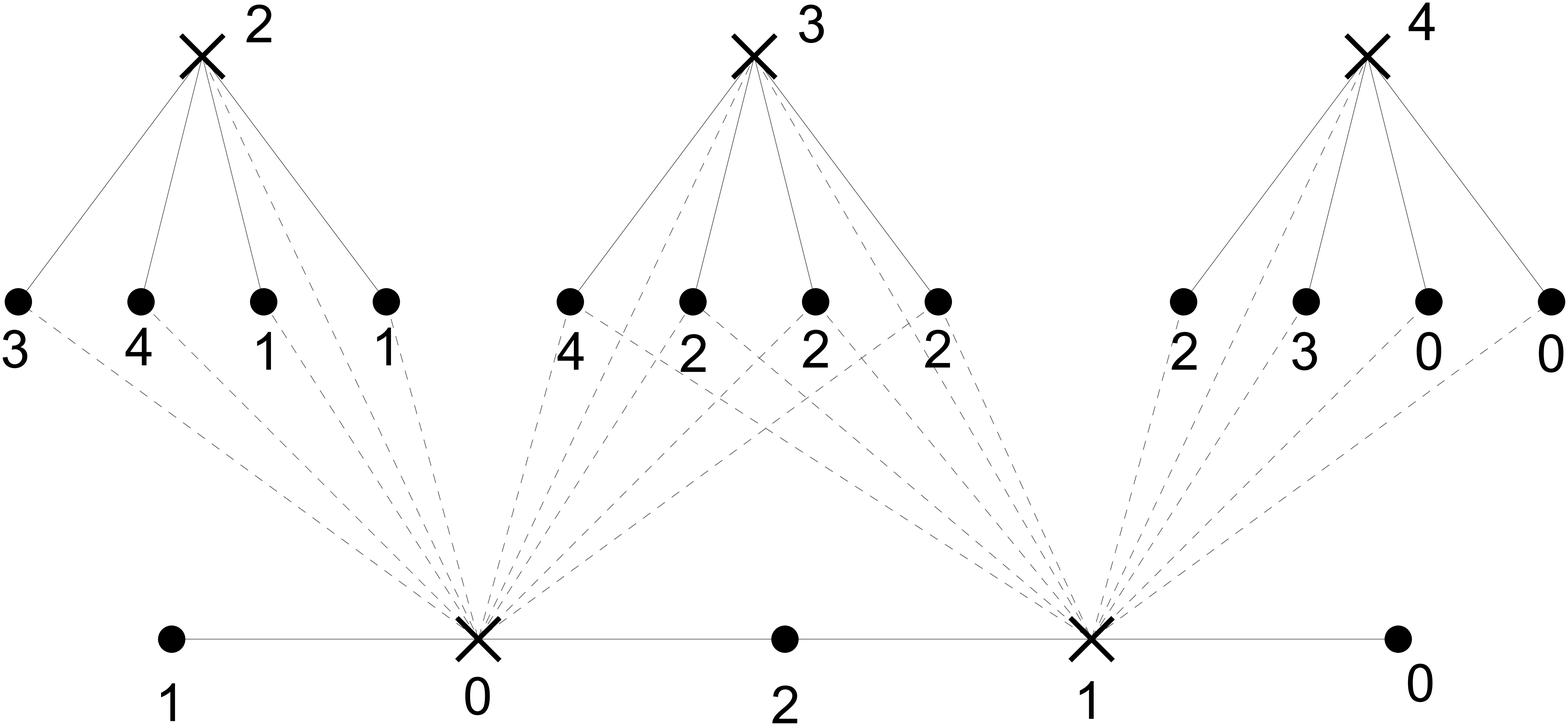}\\
\includegraphics[scale=0.12]{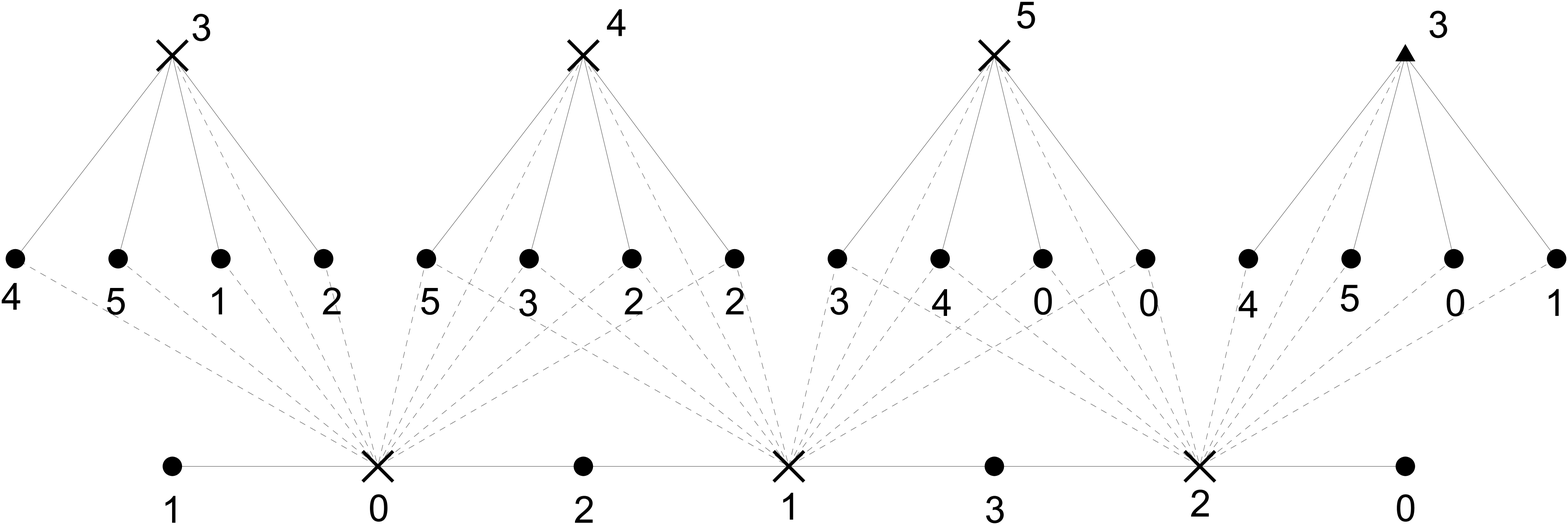}\\
\includegraphics[scale=0.1]{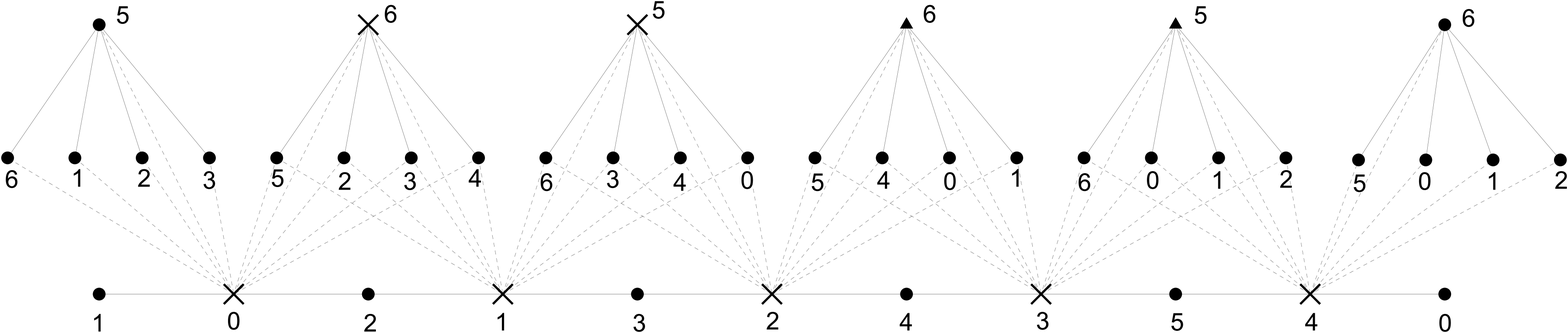}
\end{tabular}
\caption{Optimal $b$-chromatic coloring of $P_n\boxdot S_4$, for all $n\in\{3,4,6\}$.}
\label{Fig_P3S4}
\end{center}
\end{figure}

\begin{figure}[ht]
\begin{center}
\includegraphics[scale=0.07]{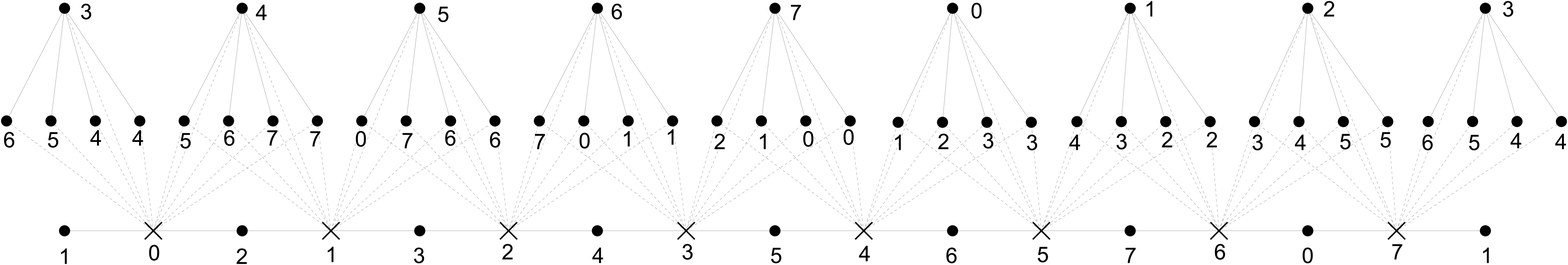}
\caption{Optimal $b$-chromatic coloring of $P_9\boxdot S_4$.}
\label{Fig_P9S4}
\end{center}
\end{figure}

\vspace{0.2cm}

Finally, we study the SVN corona of a path and a complete graph.

\begin{theorem}\label{theorem_PnKt} Let $n>2$ and $t$ be two positive integers. Then,
\[\varphi(P_n \boxdot K_t)=
\begin{cases}
\begin{array}{ll}
t+2, & \text{ if } n \leq t+3,\\
n-1,&\text{ if } t+3<n\leq 2t+3,\\
2t+3,& \text{ otherwise}.
\end{array}
\end{cases}\]
\end{theorem}

\begin{proof} Again, the case $n>2t+3$ follows from Proposition \ref{proposition_2t4}. So, we assume from now on that $n\leq 2t+3$. From Lemma \ref{lemma1} and (\ref{eq_d}), all the described values are upper bounds of the $b$-chromatic number under consideration. In order to see that they are reached, we define an appropriate $b$-chromatic coloring $c$ of $P_n\boxdot K_t$ satisfying (\ref{eq_Pn}). For each pair of non-negative integers $i<n$ and $j<t$, the following two cases arise.

If $n\leq t+3$, then $c(v_{i,j})=(i+j+1)\,\mathrm{mod}\,(t+2)$. A $b$-rainbow set is formed by the vertices $s_{0,1},\ldots, s_{n-2,n-1}$, together the vertices $v_{1,n-3},\ldots,v_{1,t-1}$. (Figure \ref{Fig_P6K4} illustrates the graph $P_6\boxdot K_4$.)

If $t+3<n\leq 2t+3$, then
    \[c(v_{i,j})=\begin{cases}
    \begin{array}{ll}
    (i+2j+3)\,\mathrm{mod}\, (n-1),& \text{ if } j< \left\lfloor\frac {n-4}2\right\rfloor,\\
    c\left(v_{((i+1)-i\,\mathrm{mod}\ 2)\,\mathrm{mod}\,(n-1),\,j-\left\lfloor\frac {n-4}2\right\rfloor}\right), & \text{ if }  \left\lfloor\frac {n-4}2\right\rfloor<j<t-1,\\
    (i+1)\,\mathrm{mod}\,(n-1),& \text{ if } j=t-1=1,\\
    (i-2)\,\mathrm{mod}\,(n-1),& \text{ if } j=t-1\neq 1.
    \end{array}
    \end{cases}\]
A $b$-rainbow set is formed by the vertices $s_{0,1},\ldots,s_{n-2,n-1}$. (Figure \ref{Fig_P6K2} illustrates the case $(n,t)\in\{(5,1),\,(6,2),\,(9,4)\}$.)
\end{proof}

\begin{figure}[ht]
\begin{center}
\includegraphics[scale=0.1]{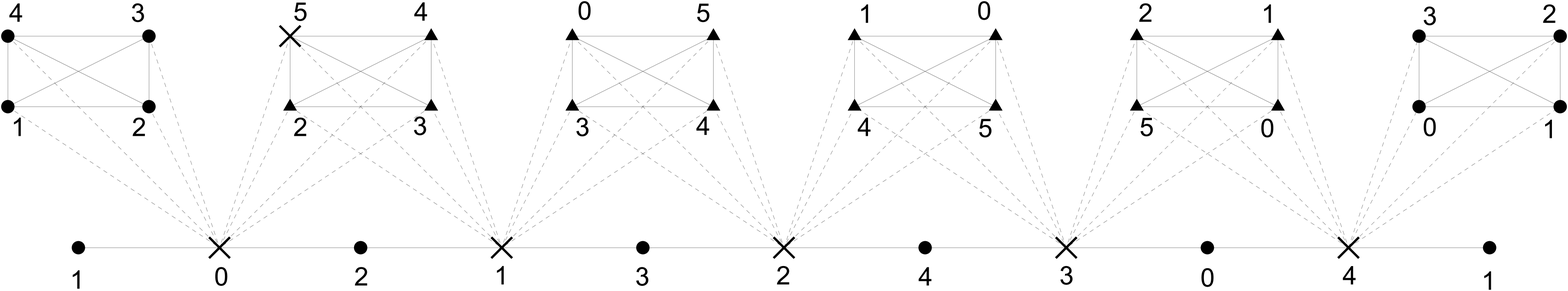}
\caption{Optimal $b$-chromatic coloring of $P_6\boxdot K_4$.}
\label{Fig_P6K4}
\end{center}
\end{figure}

\begin{figure}[ht]
\begin{center}
\begin{tabular}{c}
\includegraphics[scale=0.125]{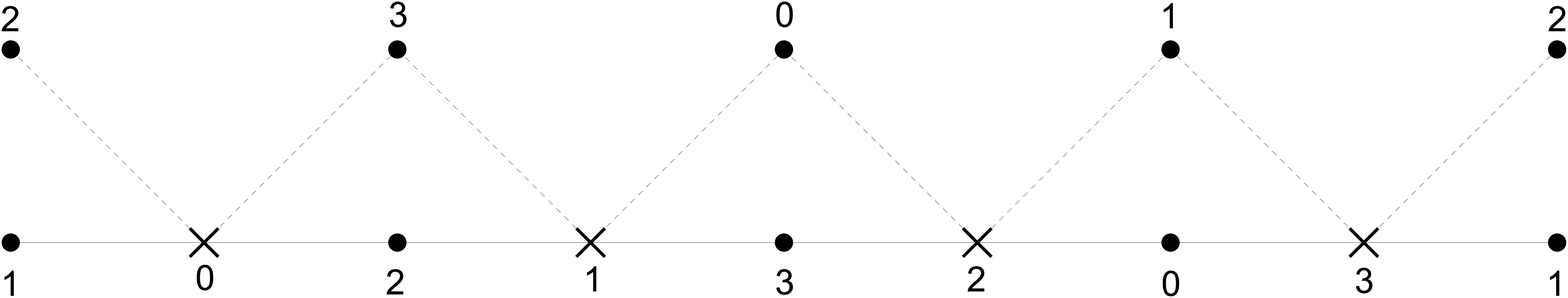}\\
\includegraphics[scale=0.1]{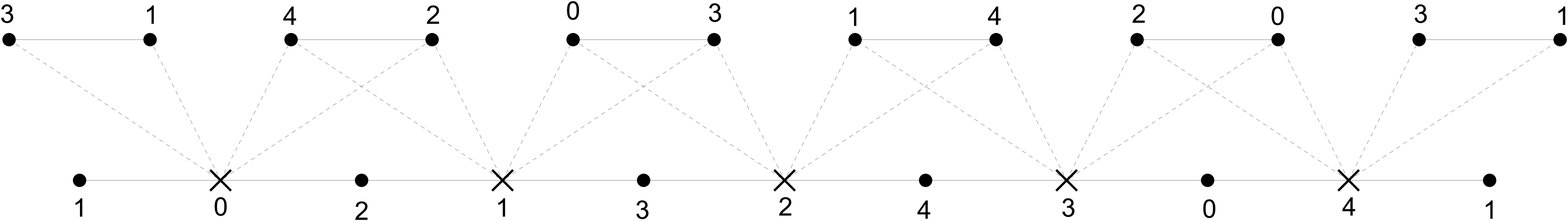}\\
\includegraphics[scale=0.0653]{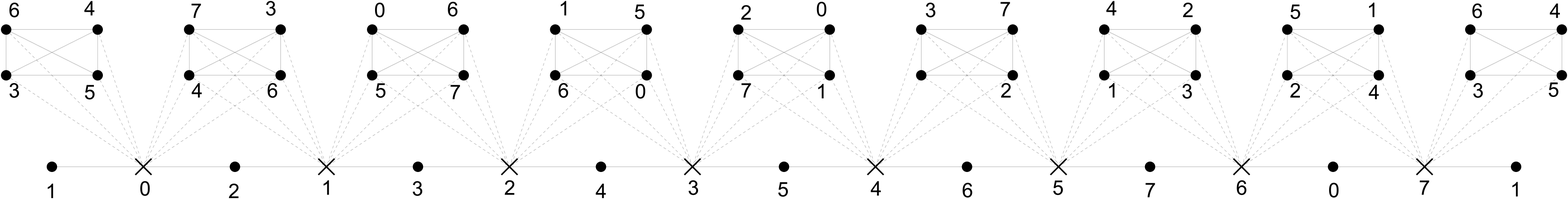}
\end{tabular}
\caption{Optimal $b$-chromatic colorings of $P_5\boxdot K_1$, $P_6\boxdot K_2$ and $P_9\boxdot K_4$.}
\label{Fig_P6K2}
\end{center}
\end{figure}

\section{SVN corona of cycles}\label{sec:cycle}

In this section, we determine the $b$-chromatic number of the SVN corona $C_n\boxdot G$ of a cycle $C_n=\langle\,u_0,\ldots,u_{n-1},\,u_0\,\rangle$, with $n>2$, and a graph $G\in\mathcal{G}$.  As a preliminary result, Proposition \ref{prop_2t4} enables us to study this number for the SVN corona $C_n\boxdot H$, for any arbitrary graph $H$ such that $n\geq \Delta(H)+3$.

\begin{proposition}\label{proposition_Cn_2t4} The following statements hold.
\begin{enumerate}
\item[(a)] $\varphi(C_n\boxdot H)\leq n$, whenever $\Delta(H)+3\leq n\leq 2\cdot |V(H)|+2$.

\vspace{0.2cm} 

\item[(b)] $\varphi(C_n\boxdot H)=2\cdot |V(H)|+3$, whenever $n>2\cdot |V(H)|+2$.
\end{enumerate}
\end{proposition}

\begin{proof}  The first statement follows readily from (b.2) in Proposition \ref{prop_2t4} once it is observed that $\Delta(C_n)=2$ and $|I(C_n)|=n$. Further, since $\Delta(C_n)=2<|V(H)|+3$, Proposition \ref{prop_2t4} implies that $\varphi(C_n\boxdot H)\leq 2\cdot |V(H)|+3$. In order to prove that this upper bound is reached, it is enough to consider the $b$-chromatic coloring of the graph $C_n\boxdot H$ that is defined as the coloring $c$ in the proof of Proposition \ref{prop_2t4}, except for
\[c(s_{0,n-1})=\begin{cases}
\begin{array}{ll}
(n-1)\,\mathrm{mod}\,(2\cdot|V(H)|+3),& \text{ if } n\not\equiv 2\,(\mathrm{mod}\,(2\cdot|V(H)|+3)),\\
0,&\text{ otherwise}.
\end{array}
\end{cases}\]
for every pair of non-negative integers $i<n$ and $k<|V(H)|$. A $b$-rainbow set is formed by the vertices $s_{0,1},\ldots,s_{2\cdot|V(H)|+2,2\cdot|V(H)|+3}$.
\end{proof}

\vspace{0.2cm}

Figure \ref{Fig_C10P3} illustrates Proposition \ref{proposition_Cn_2t4} for the graphs $C_{10}\boxdot P_3$ and $C_{11}\boxdot P_3$.

\begin{figure}[ht]
\begin{center}
\includegraphics[scale=0.08]{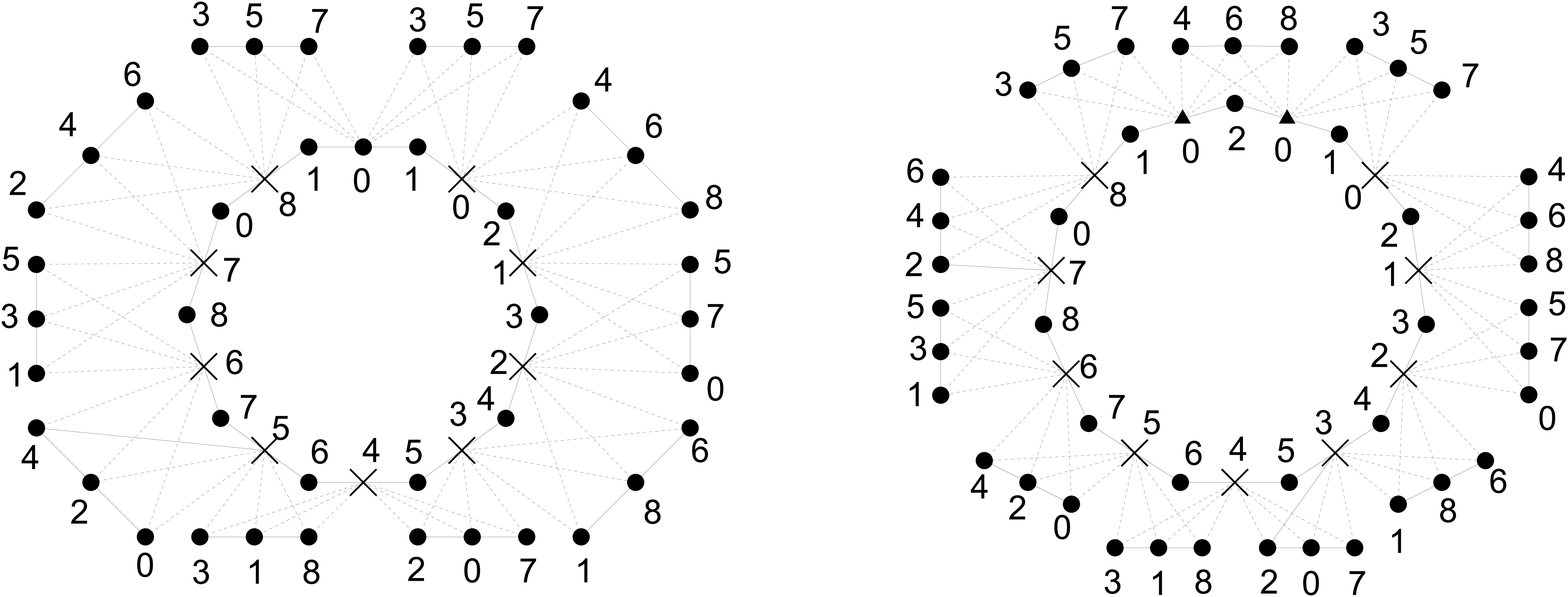}
\caption{Optimal $b$-chromatic colorings of $C_{10}\boxdot P_3$ and $C_{11}\boxdot P_3$.}
\label{Fig_C10P3}
\end{center}
\end{figure}

Now, we focus separately on each one of the mentioned graphs $C_n\boxdot G$, with $G\in\mathcal{G}$. In all the proofs, we define an appropriate  $b$-chromatic coloring $c$ of the graph $C_n\boxdot G$ satisfying (\ref{eq_Pn}) and
\begin{equation}\label{eq_Cn}
c(s_{0,n-1})=\begin{cases}
\begin{array}{ll}
(n-1)\,\mathrm{mod}\,m(C_n\boxdot G),& \text{ if } n\not\equiv 2\,(\mathrm{mod}\,m(C_n\boxdot G)),\\
0,&\text{ otherwise}.
\end{array}
\end{cases}
\end{equation}
We start by determining the $b$-chromatic number of the SVN corona of a cycle and a path.

\begin{theorem}\label{theorem_CtPn} Let $n>2$ and $t>2$ be two positive integers. Then,
\[\varphi(C_n \boxdot P_t)=\begin{cases}
\begin{array}{ll}
5,&\text{ if }  n\in\{3,4\},\\
n,  &\text{ if } 5\leq n\leq 2t+2,\\
2t+3,&\text{ otherwise}.
\end{array}
\end{cases}\]
\end{theorem}

\begin{proof}
The case $n>2t+2$ follows from Proposition \ref{proposition_Cn_2t4}. So, we assume from now on that $n\leq 2t+3$. From Lemma \ref{lemma1} and (\ref{eq_d}), all the described values are upper bounds of the $b$-chromatic number under consideration. In order to see that they are reached, we define an appropriate $b$-chromatic coloring $c$ of the graph $C_n\boxdot P_t$ satisfying (\ref{eq_Pn}) and (\ref{eq_Cn}). For each pair of non-negative integers $i<n$ and $j<t$, the following two cases arise. Here, we assume that $P_t=\langle\,v_0,\ldots,v_{t-1}\,\rangle$. 

If $n\in\{3,4\}$, then
 \[c(v_{i,j})=\begin{cases}
    \begin{array}{ll}
    0,&\text{ if } (i,j)\in\{(2,2),\,(3,0)\},\\
    1,&\text{ if } (i,j)\in\{(0,0),\,(3,1)\},\\
    2,&\text{ if } (i,j)=(1,0),\\
    3,& \text{ if } (i,j)\in\{(1,1),\,(2,0)\},\\
    4,& \text{ if } (i,j)\in\{(0,1),\,(1,2),\,(2,1)\},\\
    c(v_{i,j-2}),& \text{ otherwise}.
    \end{array}
    \end{cases}\]   
 A $b$-rainbow set is formed by the vertices $s_{0,1},\ldots,s_{n-2,n-1}, v_{1,1} ,v_{2,1}$, together with the vertex $s_{0,2}$, if $n=3$. (Figure \ref{Fig_C3P3} illustrates the graphs $C_3\boxdot P_3$ and $C_4\boxdot P_4$.)

\begin{figure}[ht]
\begin{center}
\includegraphics[scale=0.09]{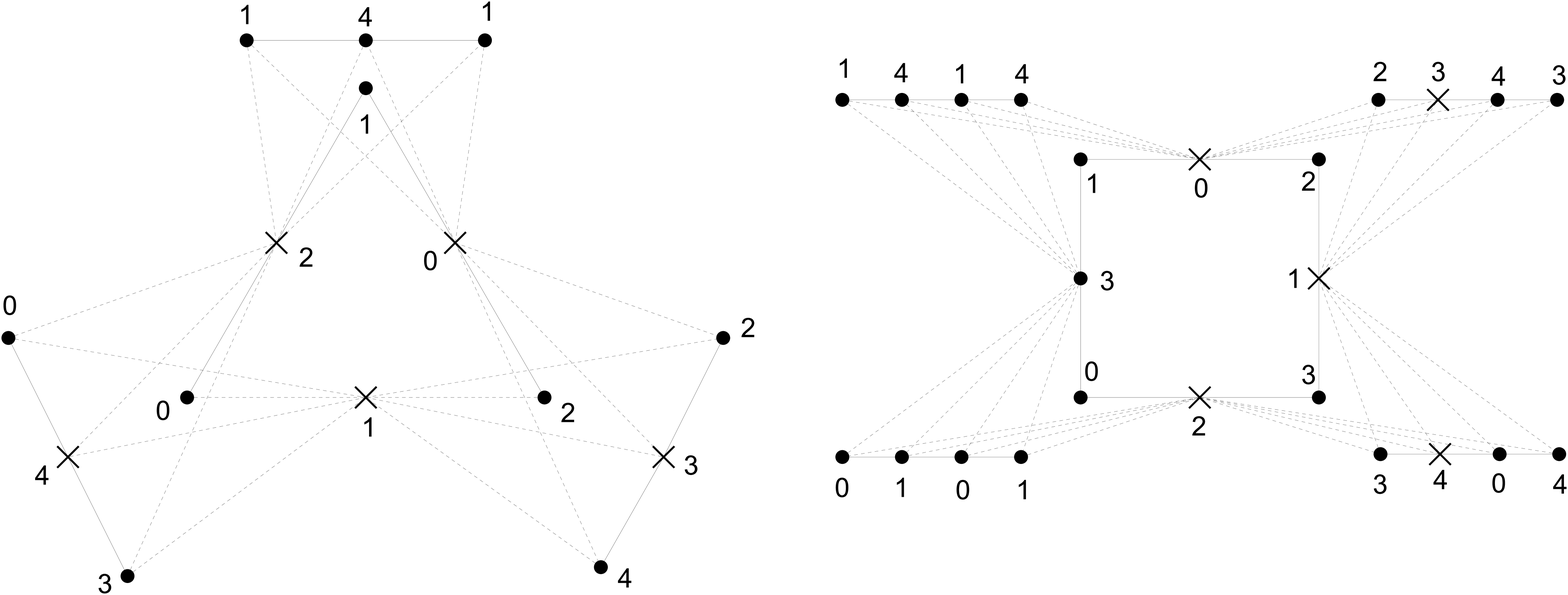}
\caption{Optimal $b$-chromatic colorings of $C_3\boxdot P_3$ and $C_4\boxdot P_4$.}
\label{Fig_C3P3}
\end{center}
\end{figure}
    
Further, if $5\leq n\leq 2t+2$, then $c(v_{i,j})$ is
         \[\begin{cases}
    \begin{array}{ll}
 i+1, & \text{ if }  (j,n)=(1,5),\\
 (i+2j+3)\,\mathrm{mod}\,n, & \text{ if }  n \text{ is odd and }j<\frac {n-3}2,\\
 
        \left(i+2j+(4-2(i\,\mathrm{mod}\,2))\right)\,\mathrm{mod}\,(n-1), & \text{ if } \begin{cases}
        n, i \text{ are even}, \text{ and } j<\left\lfloor\,\frac {n-3}2\,\right\rfloor,\\
        n, i+1 \text{ are even, and } j<\left\lceil\,\frac {n-3}2\,\right\rceil,
        \end{cases}\\
        (i+2)\,\mathrm{mod}\,6,&\text{ if } (j,n)=(1,6) \text{ and } i \text{ is even},\\
        c(v_{i,j-2}),&\text{ otherwise}.
    \end{array}
    \end{cases}\]
    A $b$-rainbow set is formed by the vertices $s_{0,1},\ldots,s_{n-2,n-1}, s_{0,n-1}$. (Figure \ref{Fig_C5P3} illustrates the graphs $C_5\boxdot P_3$ and $C_6\boxdot P_3$.)
\end{proof}

\begin{figure}[ht]
\begin{center}
\includegraphics[scale=0.11]{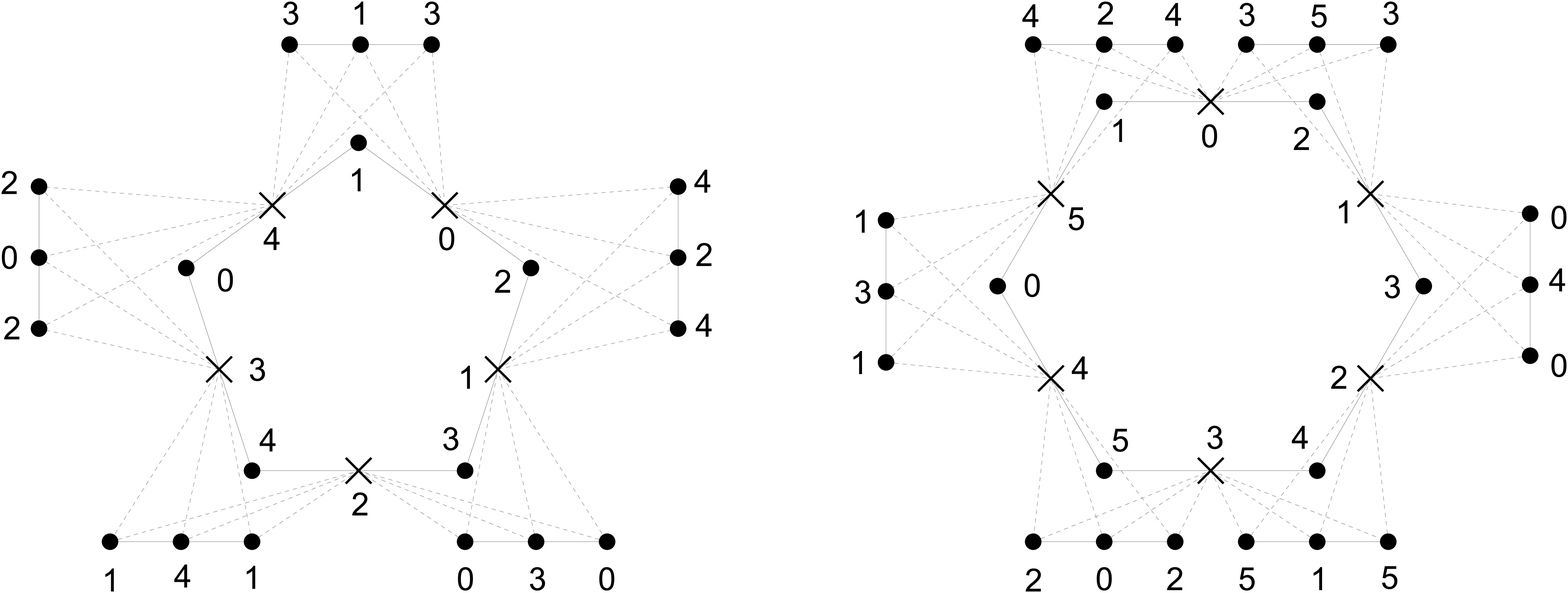}
\caption{Optimal $b$-chromatic colorings of $C_5\boxdot P_3$ and $C_6\boxdot P_3$.}
\label{Fig_C5P3}
\end{center}
\end{figure}

The next graph to study is the SVN corona of two cycles.

\begin{theorem}\label{theorem_CnCt} Let $n>2$ and $t>2$ be two positive integers. Then,
\[\varphi\left(C_n \boxdot C_t\right) =
\begin{cases}
\begin{array}{ll}
5,&\text{ if }  n\in\{3,4\},\\
n,  &\text{ if } 5\leq n\leq 2t+2,\\
2t+3,&\text{ otherwise}.
\end{array}
\end{cases}\]
\end{theorem}

\begin{proof}
The case $n>2t+2$ follows from Proposition \ref{proposition_Cn_2t4}. So, we assume from now on that $n\leq 2t+2$. From Lemma \ref{lemma1} and (\ref{eq_d}), all the described values are upper bounds of the $b$-chromatic number under consideration. In order to see that they are reached, it is enough to consider the same map $c$ defined in the proof of Theorem \ref{theorem_CtPn}, except for

 \[c(v_{i,j})=\begin{cases}
    \begin{array}{ll}
    2,&\text{ if } (i,j,n)=(0,2,4),\\
    4,&\text{ if } (i,j,n)=(3,2,4),\\
    3,&\text{ if } (i,j,n)=(0,2,3),\\
    (i+2)\,\mathrm{mod}\,n, &\text{ if } j=2 \text{ and } n\in\{5,7\},\\ 
    c(v_{(i+(-1)^{i\,\mathrm{mod}\,2})\,\mathrm{mod}\,n,0}), & \text{ if } j=2 \text{ and } n\in\{6,8\}.
    \end{array}
    \end{cases}\]
    Here, we have assumed that $C_t=\langle\,v_0,\ldots,v_{t-1},\,v_0\,\rangle$. The same $b$-rainbow sets described in the proof of Theorem \ref{theorem_CtPn} are valid here. (Figure \ref{Fig_C5C3} illustrates the graphs $C_5\boxdot C_3$ and $C_6\boxdot C_3$.)
\end{proof}

\begin{figure}[ht]
\begin{center}
\includegraphics[scale=0.1]{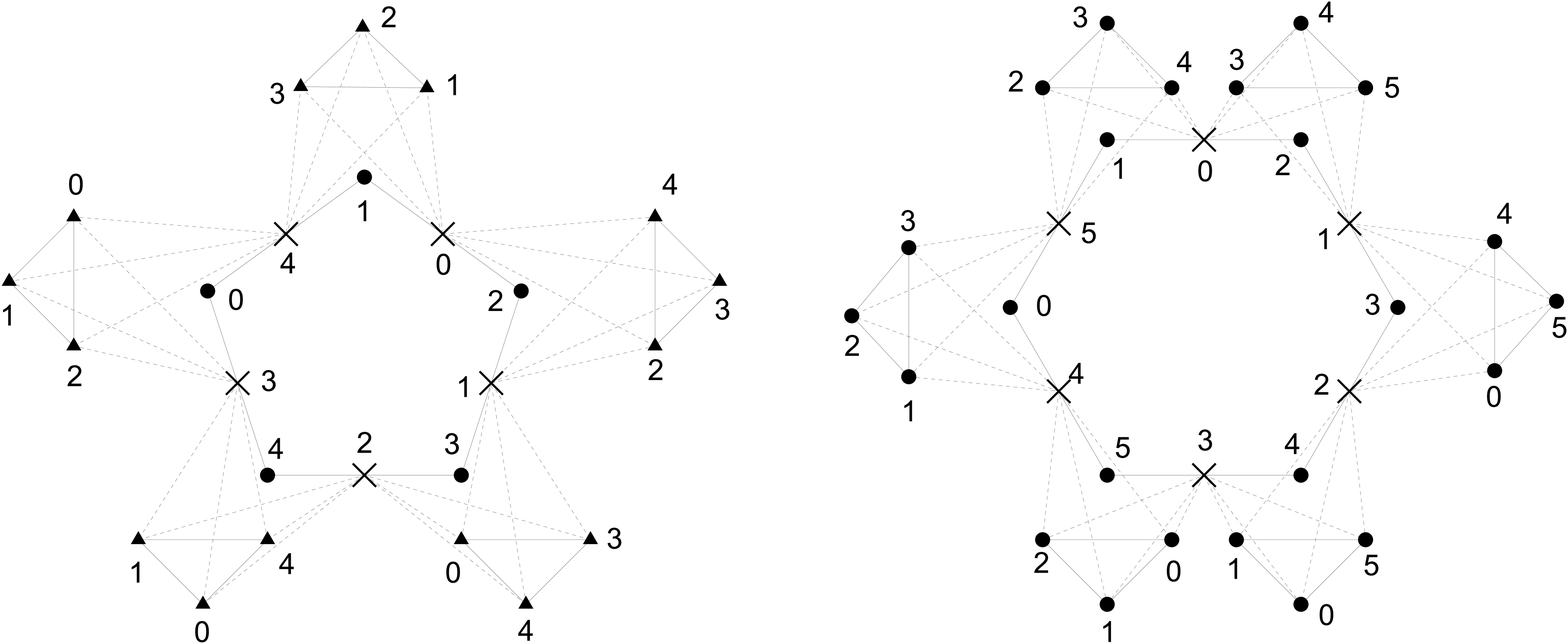}
\caption{Optimal $b$-chromatic colorings of $C_5\boxdot C_3$ and $C_6\boxdot C_3$.}
\label{Fig_C5C3}
\end{center}
\end{figure}

Now, we study the SVN corona of a cycle and a star.  (Again, recall that $S_t$ is the star of order $t+1$.)

\begin{theorem}\label{theorem_CnSt}
Let $n>2$ and $t>2$ be two positive integers. Then,
\[\varphi(C_n\boxdot S_t)=\begin{cases}
\begin{array}{ll}
2n,&\text{ if } n\leq \lfloor\frac {t+3}2\rfloor,\\
t+3,&\text{ if } \lfloor\frac {t+3}2\rfloor< n \leq t+2,\\
n,&\text{ if } t+3\leq n\leq 2t+4,\\
2t+5, & \text{ otherwise}.
\end{array}
\end{cases}\]
\end{theorem}

\begin{proof}  The case $n>2t+4$ follows from Proposition \ref{proposition_Cn_2t4}. So, we assume that $n\leq 2t+4$. From Lemma \ref{lemma1} and (\ref{eq_d}), all the described values are upper bounds of the $b$-chromatic number. To see that they are reached, we define the $b$-chromatic coloring $c$ of $C_n\boxdot S_t$ satisfying (\ref{eq_Pn}) and (\ref{eq_Cn}) such that, for each pair of non-negative integers $i<n$ and $j\leq t$, the following two cases arise. Here, we assume that $V(S_t)=\{v_0,\ldots,v_t\}$, where $v_0$ is the center of the star.  

If $n\leq t+2$, then let $\alpha_{n,t}=m(C_n\boxdot S_t)-n$, then
    \[c(v_{i,j})=\begin{cases}
    \begin{array}{ll}
     n+((i+j)\,\mathrm{mod}\,\alpha_{n,t}), & \text{ if } j<\alpha_{n,t},\\
    (j-\alpha_{n,t}+1)\,\mathrm{mod}\, n,& \text{ if } \alpha_{n,t}\leq j<\alpha_{n,t} + n-2 \text{ and } i=0,\\
    (j-\alpha_{n,t})\,\mathrm{mod}\, n,& \text{ if } \alpha_{n,t}\leq j<\alpha_{n,t} + n-2 \text{ and } i=n-1,\\
       (i+j-\alpha_{n,t}+1)\,\mathrm{mod}\, n,& \text{ if } \alpha_{n,t}\leq j<\alpha_{n,t} + n-3 \text{ and } i\not\in\{0,n-1\},\\
    c(v_{i,j-1}),&\text{ otherwise}.
    \end{array}
    \end{cases}\]
    A $b$-rainbow set is formed by the vertices $s_{0,1},\ldots, s_{n-2,n-1},s_{0,n-1}$, together with either the vertices $v_{0,0},\ldots, v_{n-1,0}$, if $n\leq \lfloor\frac {t+3}2\rfloor$; or the vertices $v_{1,1},\ldots,v_{t+3-n,1}$, if $\lfloor\frac {t+3}2\rfloor< n \leq t+2$.
    
Further, if $t+3\leq n\leq 2t+5$, then
 \[c(v_{i,j})=\begin{cases}
    \begin{array}{ll}
    (i+3)\,\mathrm{mod}\,n,& \text{ if } j=0,\\
    (i-j-2)\,\mathrm{mod}\,n, & \text{ if } i \text{ is even and } 0<j\leq\min\{n-5 , t\},\\
    (i+j+3)\,\mathrm{mod}\,n, & \text{ if } i \text{ is odd and } 0<j\leq\min\{n-5 , t\},\\
    c(v_{i,j-1}), & \text{ otherwise}.
    \end{array}
    \end{cases}\]
    A $b$-rainbow set is formed by the vertices $s_{0,1},\ldots, s_{n-2,n-1},s_{0,n-1}$.
\end{proof}

\vspace{0.2cm}

Figure \ref{Fig_C3S3} illustrates Theorem \ref{theorem_CnSt} for the graphs $C_3\boxdot S_3$, $C_5\boxdot S_3$ and $C_7\boxdot S_3$.

\begin{figure}[ht]
\begin{center}
\includegraphics[scale=0.085]{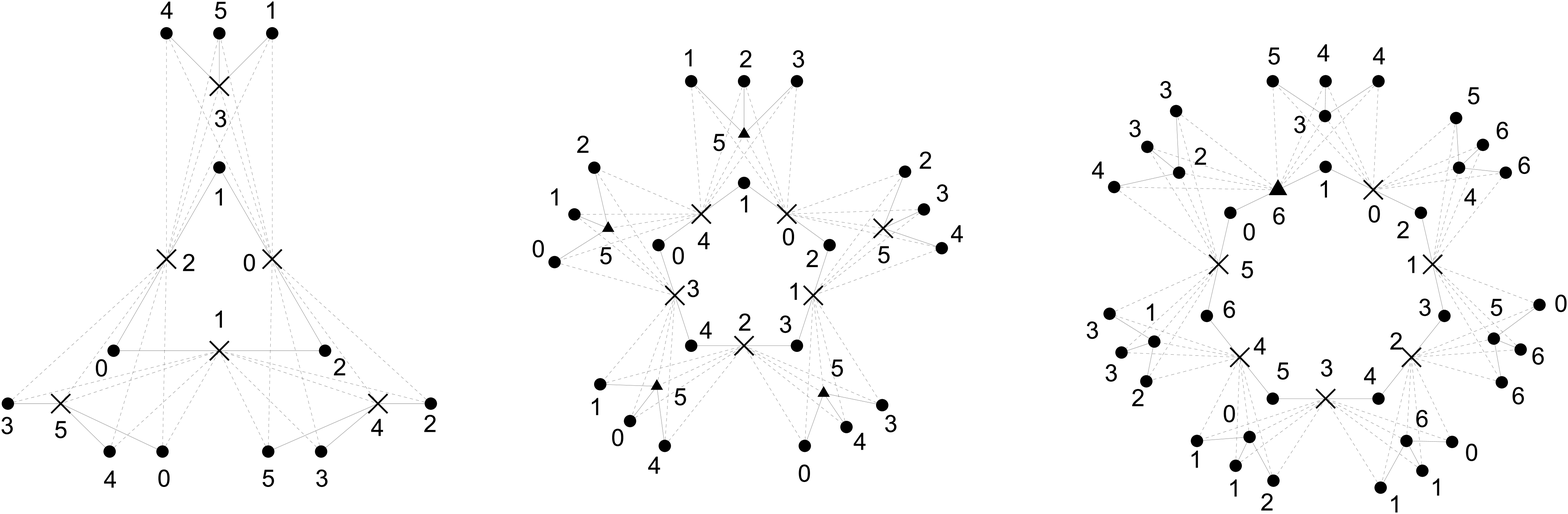}
\caption{Optimal $b$-chromatic coloring of $C_n\boxdot S_3$, for all $n\in\{3,5,7\}$.}
\label{Fig_C3S3}
\end{center}
\end{figure}

\vspace{0.5cm}

Finally, we focus on the SVN corona of a cycle and a complete graph.

\begin{theorem}\label{theorem_CnKt} Let $n>2$ and $t$ be two positive integers. Then,
\[\varphi(C_n\boxdot K_t)=\begin{cases}
\begin{array}{ll}
t+2,& \text{ if } n\leq t+1,\\
n,& \text{ if } t+2\leq n\leq 2t+3,\\
2t+3, & \text{ otherwise}.
\end{array}
\end{cases}\]
\end{theorem}

\begin{proof} The case $n>2t+3$ follows from Proposition \ref{proposition_Cn_2t4}. So, we assume that $n\leq 2t+3$. From Lemma \ref{lemma1} and (\ref{eq_d}), all the described values are upper bounds of the $b$-chromatic number. To see that they are reached, we define an appropriate $b$-chromatic coloring $c$ of  $C_n\boxdot K_t$ satisfying (\ref{eq_Pn}) and (\ref{eq_Cn}). For each pair of non-negative integers $i<n$ and $j<t$, we define $c(v_{i,j})$ as in the proof of Theorem \ref{theorem_PnKt}, except for $c(v_{0,n-2})=t$, if $n\leq t+1$. In this last case, a $b$-rainbow set is formed by the vertices $s_{0,1},\ldots, s_{n-2,n-1},\,v_{1,n-3},\ldots,v_{1,t-1}$. Otherwise, if $t+1< n\leq 2t+3$, then a $b$-rainbow set is formed by $s_{0,1},\ldots, s_{n-2,n-1},s_{0,n-1}$. (Figure \ref{Fig_C5K4} illustrates the graphs $C_5\boxdot K_4$ and $C_6\boxdot K_4$.)
\end{proof}

\begin{figure}[ht]
\begin{center}
\includegraphics[scale=0.1]{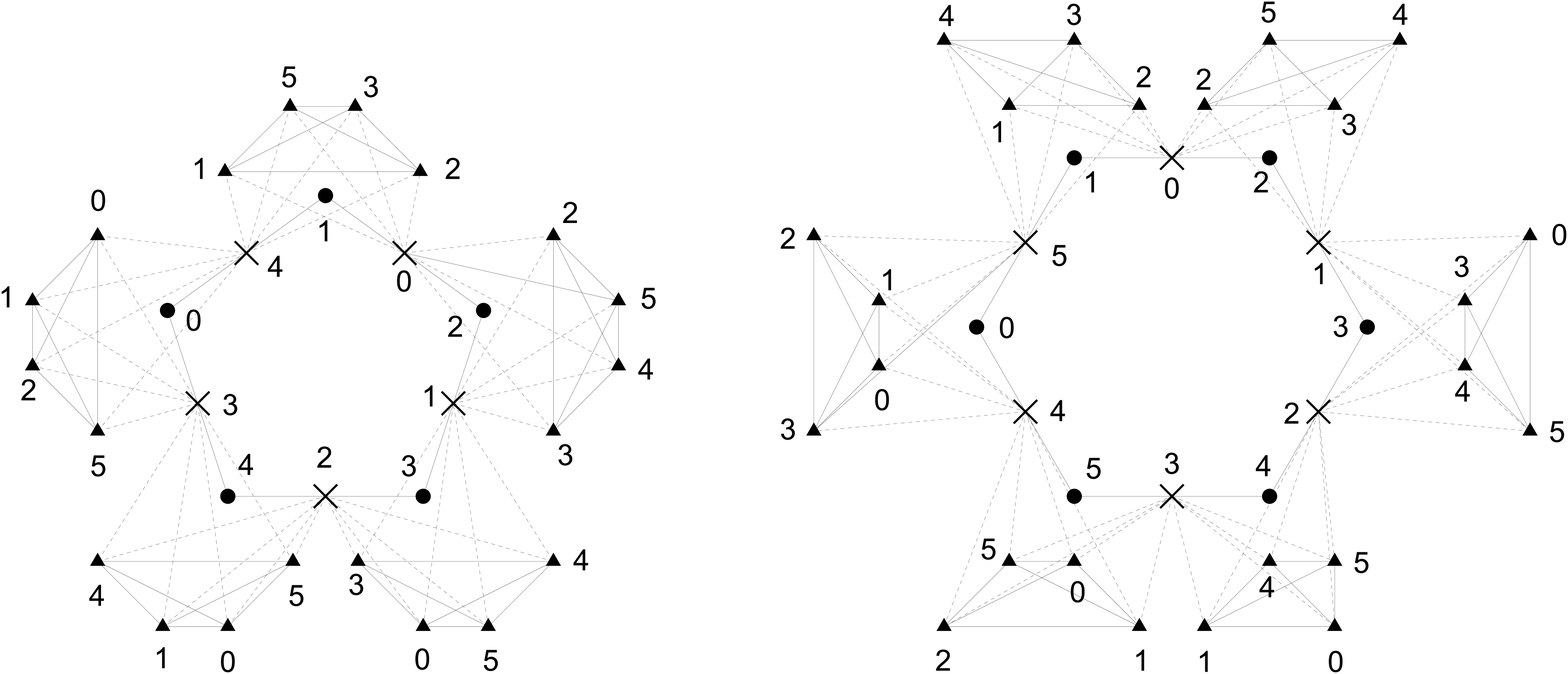}
\caption{Optimal $b$-chromatic colorings of $C_5\boxdot K_4$ and $C_6\boxdot K_4$.}
\label{Fig_C5K4}
\end{center}
\end{figure}

\section{SVN corona of stars}\label{sec:star}

In this section, we determine the $b$-chromatic number of the SVN corona of a star $S_n$, with $n>2$, of set of vertices $V(S_n)=\{u_0,\ldots,u_n\}$, where $u_0$ is the center, with a graph $G\in\mathcal{G}$.  As a preliminary result, we determine this number for the SVN corona $S_n\boxdot H$, for any arbitrary graph $H$, with $\Delta(H)+1<\min\{n,\,|V(H)|+2\}+\varphi(H)$.

\begin{lemma} \label{lemma_Sn_LB} Let $n>2$ be a positive integer and let $H$ be a graph of order $t\geq 1$ such that $\Delta(H)+1<\varphi(S_n\boxdot H)$. Then, $\varphi(S_n\boxdot H)=\min\{n,\,|V(H)|+2\}+\varphi(H)$.
\end{lemma}

\begin{proof} Since $\Delta(H)+1<\varphi(S_n\boxdot H)$, the vertex $v_{i,j}$ cannot be a $b$-vertex of $S_n\boxdot H$. Thus, every $b$-rainbow set of $S_n\boxdot H$ is formed by a subset of vertices of the form $s_{i,j}$, together with either the vertex $u_0$ or a subset of vertices of the form $v_{0,k}$. (Observe to this end that the remaining vertices $u_i$, with $0<i<n$, have degree one; and also that the vertex $u_0$ is not adjacent to any vertex of the form $v_{0,k}$.)

Then, the required result of being upper bound follows readily from Lemma \ref{lemma1} applied to the graph obtained after removing the vertices $u_0,\,v_{0,0},\ldots,v_{0,|V(H)|-1}$ from the graph $S_n\boxdot H$, together with the fact that every vertex $s_{i,j}$ is adjacent to $u_0$ and every vertex $v_{0,k}$.

Now, in order to prove that the described upper bound is reached, we define an appropriate $b$-chromatic coloring $c$ of $S_n\boxdot H$. To this end, let $\alpha_{n,|V(H)|}=\min\{n,\,|V(H)|+2\}$ and let $c':V(H)\rightarrow \{0,\ldots,\varphi(H)-1\}$ be an optimal $b$-chromatic coloring of the graph $H$. Then, we define $c(u_0)=0$ and $c(v_{0,k})=c'(v_k)$, for every non-negative integer $k<|V(H)|$. In addition, for each positive integer $i\leq n$, we define $c(s_{0,i})=\varphi(H)+((i-1)\,\mathrm{mod}\,\alpha_{n,|V(H)|})$ and $c(u_i)=\varphi(H)+(i\,\mathrm{mod}\,\alpha_{n,|V(H)|})$. Furthermore, we have from Brooks' Theorem \cite{Brook1941} that every proper coloring of the vertices $v_{i,0},\ldots,v_{i,|V(H)|-1}$ requires at least either $\Delta(H)$ or $\Delta(H)+1$ distinct colors. These vertices can always be colored by using all the colors of the set $\{\varphi(H),\ldots,\varphi(H)+\alpha_{n,|V(H)|}-1\}\setminus\{c(s_{0,i}),\,c(u_i)\}$, together with, at most, the colors $0,\ldots,\Delta(H)+2-n$ in case of being $n\leq \Delta(H)+2$.
\end{proof}

\vspace{0.2cm}

Figure \ref{Fig_S4P3} illustrates the previous result for the graphs $S_n\boxdot P_3$, with $n\in\{3,4,5\}$.

\begin{figure}[ht]
\begin{center}
\includegraphics[scale=0.065]{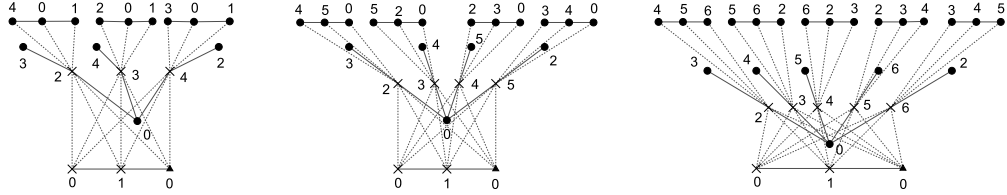}
\caption{Optimal $b$-chromatic colorings of $S_n\boxdot P_3$, for all $n\in\{3,4,5\}$.}
\label{Fig_S4P3}
\end{center}
\end{figure}

In addition, the following lemma establishes a lower bound for the $b$-chromatic number of the graph $S_n\boxdot H$, where $H$ is an arbitrary graph.

\begin{lemma} \label{lemma_Sn_LB_4} Let $n>2$ be a positive integer and let $H$ be any graph. Then, $\varphi(S_n\boxdot H)\geq 4$.
\end{lemma}

\begin{proof} It is readily verified that $\chi(S_n\boxdot H)=\chi(H)+1$. Thus, if $\chi(H)\geq 3$, then the result follows straightforwardly from Lemma \ref{lemma0}. So, we may assume that $\chi(H)\in\{1,2\}$. It is enough to prove the existence of a $b$-rainbow set of four $b$-vertices in both cases. If $\chi(H)=1$, then let $c$ be the proper $4$-coloring of $S_n\boxdot H$ that is defined so that, for every non-negative integer $i\in\{0,1,2\}$, it is
$c(s_{0,i})=i$, $c(u_i)=(i+1)\,\mathrm{mod}\,3$ and $c(v_{i,0})=(i+2)\,\mathrm{mod}\,3$. Any other vertex $v$ is colored as $c(v)=3$. Then, the set of vertices $\{s_{0,0},\,s_{0,1},\,s_{0,2},\,u_0\}$ is a $b$-rainbow set of $S_n\boxdot H$.

Furthermore, if $\chi(H)=2$, then we may assume without loss of generality that $v_{0,0}$ and $v_{0,1}$ are adjacent. In addition, let $c':V(S_n\boxdot H)\rightarrow \{0,1,2\}$ be a proper $3$-coloring of $S_n\boxdot H$. Then, let $c''$ be the proper $4$-coloring of $S_n\boxdot H$ that is defined so that, $c''(s_{0,0})=c''(u_1)=3$, and $c''(v)=c'(v)$, otherwise. As such, the set of vertices $\{s_{0,0},\,s_{0,1},\,v_{0,0},\,v_{0,1}\}$ is a $b$-rainbow set of $S_n\boxdot H$.
\end{proof}

\vspace{0.2cm}

As an immediate consequence of the previous two results, the following theorem establishes the $b$-chromatic number of the SVN corona of a star with either a path, a cycle, or a complete graph.

\begin{theorem}\label{theorem_SnPt} Let $n>2$, $t>2$ and $t'$ be three positive integers. Then,
\[\varphi(S_n\boxdot P_t)=\begin{cases}
\begin{array}{ll}
n+2, &\text{ if } \begin{cases}
t=3\text{ and } n\in\{3,4,5\},\\
t=4\text{ and } n\in\{3,4,5,6\},
\end{cases}\\
n+3, &\text{ if } n\leq t+2 \text{ and } t>4,\\
t+4,&\text{ if } \begin{cases}
t=3 \text{ and } n>5,\\
t=4 \text{ and } n>6,
\end{cases}\\
t+5, &\text{ if } n>t+2>6.
\end{array}
\end{cases}\]

\[\varphi(S_n\boxdot C_t)=\begin{cases}
\begin{array}{ll}
n+2, & \text{ if } n=4,\\
n+3, & \text{ if } n\leq t+2, \text{ with } n\neq 4,\\
t+5, & \text{ if } n>t+2.
\end{array}
\end{cases}\]
and
\[\varphi(S_n\boxdot K_{t'})=\min\{n,t'+2\}+t'.\]
\end{theorem}

\begin{proof} The respective $b$-chromatic numbers of both graphs $S_n\boxdot P_t$ and $S_n\boxdot C_t$ follow straightforwardly from Proposition \ref{proposition_Kouider} and Lemmas \ref{lemma_Sn_LB} and \ref{lemma_Sn_LB_4}  once it is observed that $\Delta(P_t)=\Delta(C_t)=2$. Furthermore, we have from Proposition \ref{proposition_Kouider} and Lemma \ref{lemma_Sn_LB} that $\varphi(S_n\boxdot K_{t'})=\min\{n,t'+2\}+t'$, whenever $t'<\min\{n,t'+2\}+t'$. That is, it always holds.
\end{proof}

\vspace{0.2cm}

Now, we determine the $b$-chromatic number of the SVN corona of two stars.

\begin{theorem}\label{theorem_SnSt} Let $n>2$ and $t>2$ be two positive integers. Then,
\[\varphi(S_n\boxdot S_t)=\begin{cases}
\begin{array}{ll}
2n+1,& \text{ if } n\leq \frac {t+1}2,\\
t+2,& \text{ if } \frac {t+1}2<n<t,\\
\min\{n,t+3\}+2, & \text{ otherwise}.
\end{array}
\end{cases}\]
\end{theorem}

\begin{proof} Since $\Delta(S_t)=t$, we have from Proposition \ref{proposition_Kouider} and Lemma \ref{lemma_Sn_LB} that $\varphi(S_n\boxdot S_t)=\min\{n,t+3\}+2$, whenever $t+1<\min\{n,t+3\}+2$. That is, whenever $n\geq t$. So, we assume from now on that $n<t$. From Lemma \ref{lemma1} and (\ref{eq_d}), all the described values are upper bounds of the $b$-chromatic number under consideration. In order to see that they are reached, we define an appropriate $b$-chromatic coloring $c$ of the graph $S_n\boxdot S_t$. Here, we assume that $V(S_t)=\{v_0,\ldots,v_t\}$, where $v_0$ is the center of the star. For each pair of positive integers $i<n$ and $j<t$, the following two cases arise.

If $n\leq \frac{t+1}2$, then we define $c(u_i)=c(u_0)=c(v_{0,0})=2n$, $c(s_{0,i})=2(i-1)$,  $c(v_{0,j})=1+((j-1)\,\mathrm{mod}\,n)$, $c(v_{i,0})=2i-1$ and $c(v_{i,j})=2i+((j-1)\,\mathrm{mod}\,(2n-1))$. (Figure \ref{Fig_S3S4} (left) illustrates the graph $S_3\boxdot S_6$.)
    
Further, if $\frac {t+1}2<n<t$, then we define $c(u_i)=c(u_0)=c(v_{0,0})=1$, $c(s_{0,i})=2(i-1)\,\mathrm{mod}\,(t+2)$,  $c(v_{0,j})=3$, $c(v_{i,0})=(2i-1)\,\mathrm{mod}\,(t+2)$ and $c(v_{i,j})=2i+((j-1)\,\mathrm{mod}\,(t+2))$.  (Figure \ref{Fig_S3S4} (right) illustrates the graph $S_3\boxdot S_4$.)
\end{proof}

\begin{figure}[ht]
\begin{center}
\includegraphics[scale=0.08]{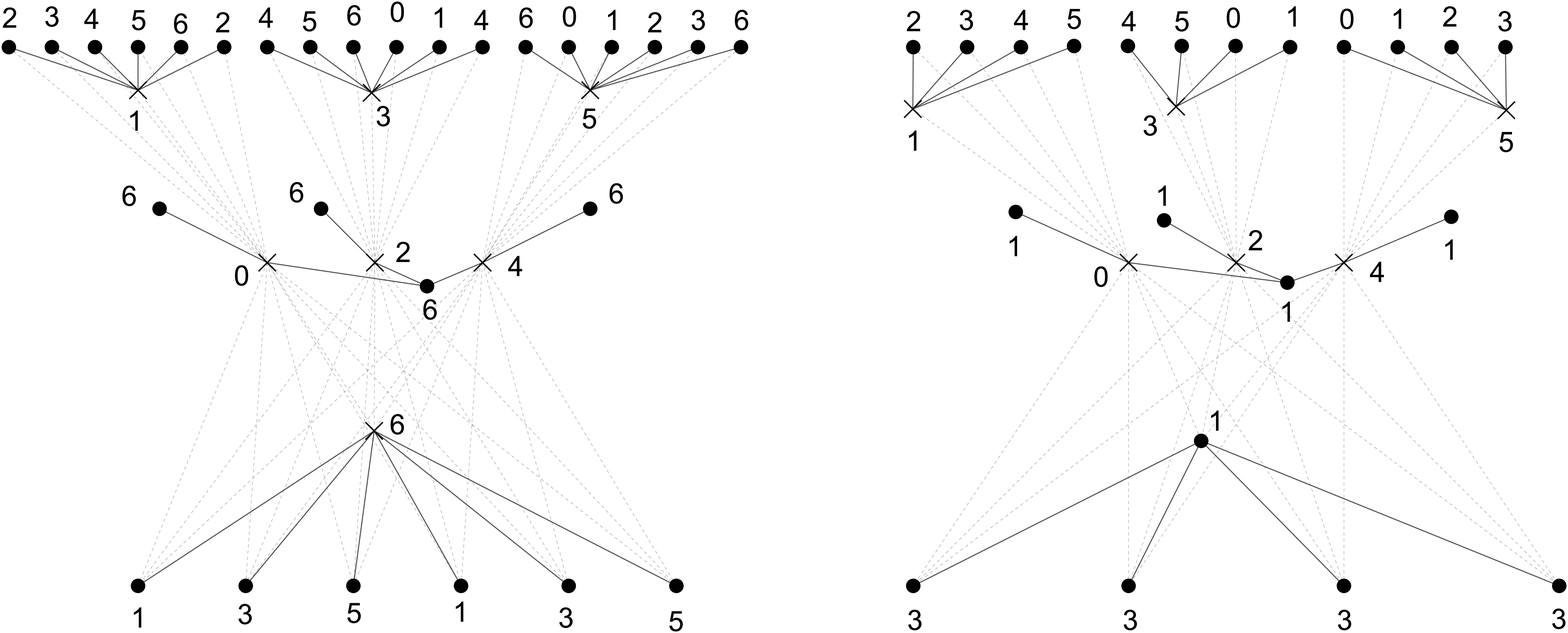}
\caption{Optimal $b$-chromatic colorings of $S_3\boxdot S_6$ and $S_3\boxdot S_4$.}
\label{Fig_S3S4}
\end{center}
\end{figure}

\section{SVN corona of complete graphs}\label{sec:complete}

In this section, we study the $b$-chromatic number of the SVN corona $K_n\boxdot G$ of a complete graph $K_n$ of set of vertices $V(K_n)=\{u_0,\ldots,u_{n-1}\}$ and a graph $G\in\mathcal{G}$, whenever $m(K_n\boxdot G)\leq n+2$. Here, we assume that $n>1$. Otherwise, $K_1\boxdot G=K_1$. The following result establishes a lower bound for a general graph $K_2\boxdot H$.

\begin{lemma}\label{lemma_Kn_n2} Let $H$ be a non-empty graph. Then, $\varphi(K_2\boxdot H)\geq \varphi(H)+1$.
\end{lemma}

\begin{proof} Let $c:V(H)\rightarrow \{0,\ldots,\varphi(H)-1\}$ be a $b$-chromatic coloring of the graph $H$. If $V(H)=\{v_0,\ldots, v_{t-1}\}$, then we define the $b$-chromatic coloring $c'$ of the graph $G\boxdot H$ such that $c'(u_0)=c'(u_1)=0$,  $c'(s_{0,1})=\varphi(H)$ and 
$c'(v_{i,j})=c(v_j)$, for all $i\in\{0,1\}$ and $j<t$. Hence, $\varphi(K_2\boxdot H)\geq \varphi(H)+1$.
\end{proof}

\vspace{0.2cm}

Figure \ref{Fig_K2P4}, together with Proposition \ref{proposition_Kouider}, shows that the lower bound described in the previous lemma is tight, but the equality does not hold in general. It is so that $\varphi(K_2\boxdot P_3)=2=\varphi(P_3)+1$, but $\varphi(K_2\boxdot P_4)=4>3=\varphi(P_4)+1$.

\begin{figure}[ht]
\begin{center}
\includegraphics[scale=0.095]{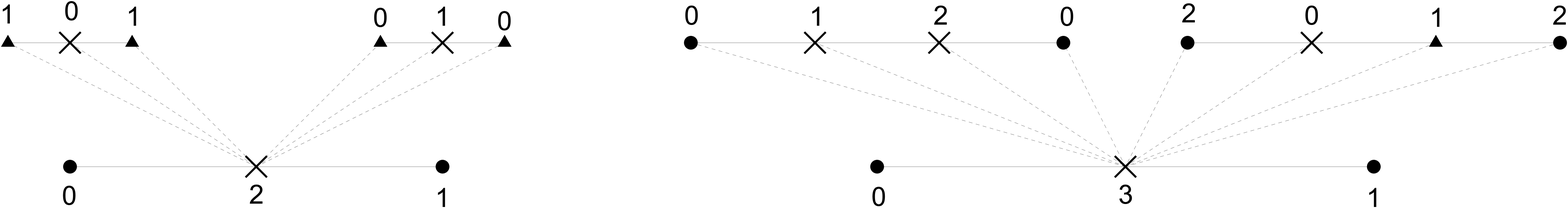}
\caption{Optimal $b$-chromatic colorings of $K_2\boxdot P_3$ and $K_2\boxdot P_4$.}
\label{Fig_K2P4}
\end{center}
\end{figure}

We study separately each one of the mentioned graphs $K_n\boxdot G$, with $G\in\mathcal{G}$. Firstly, we determine the $b$-chromatic number of the SVN corona $K_n\boxdot P_t$ in case of being $m(K_n\boxdot P_t)\leq n+2$.  From Lemma \ref{lemma1} and (\ref{eq_d}), it is equivalent to say that
either $n\in\{2,3,4\}$ or $n\geq 2t+1$.

\begin{theorem}\label{theorem_KnPt} Let $n>1$ and $t>2$ be two positive integers. Then,
\[\varphi(K_n\boxdot P_t)=\begin{cases}
\begin{array}{lll}
n+1, & \text{ if } \begin{cases}
n=2 \text{ and } t=3,\\
 n\geq 7 \text{ and } t=3,
\end{cases}\\
n+2, & \text{ if } \begin{cases}
n=2 \text{ and } t>3,\\
n\in\{3,4\},\\
n\geq 2t+1>7.
\end{cases}
\end{array}
\end{cases}\]
\end{theorem}

\begin{proof} The case $n=3$ follows from Theorem \ref{theorem_CtPn}. So, we assume that $n\neq 3$. Except for the case $(n,t)=(7,3)$, all the described values coincide with $m(K_n\boxdot P_t)$ and hence, from Lemma \ref{lemma1}, they are upper bounds of the $b$-chromatic number under consideration. Proposition \ref{proposition_Kouider} and Lemma \ref{lemma_Kn_n2} imply that this upper bound is reached in case of being $n=2$ and $t\neq 4$. In addition, Figure \ref{Fig_K2P4} (right) illustrates the case $t=4$. 

Further, even if $m(K_7\boxdot P_3)=9$, this lower bound is not reached, because every $b$-rainbow set of a $b$-chromatic coloring of $K_7\boxdot P_3$ with nine colors would only contain non-adjacent vertices of the form $s_{i,j}$ or $v_{k,1}$. A simple study of cases enables us to ensure that this condition is not feasible and hence, $\varphi(K_7\boxdot P_3)\leq 8$. This new bound is indeed reached, as we prove later for the case $n\geq 7$ and $t=3$.

Now, in order to prove that the remaining values are reached, we define an appropriate $b$-chromatic coloring $c$ of the graph $K_n\boxdot P_t$ such that $c(u_i)=i\,\mathrm{mod}\,m(K_n\boxdot G)$, for every non-negative integer $i<n$. In addition, for every non-negative integers $i,j<n$, with $i<j$, and $k<t$, the following cases arise. Here, we assume that $P_t=\langle\,v_0,\ldots,v_{t-1}\,\rangle$. 

If $n=4$, then
\[c(s_{i,j})=\begin{cases}
(i-1)\,\mathrm{mod}\,4,& \text{ if }j=i+1,\\
4 + i,&\text{ otherwise}.
\end{cases}\]
In addition,
\[c(v_{i,j})=\begin{cases}
\begin{array}{ll}
4, &\text{ if } (i,j)\in\{(1,1),(3,1)\},\\
5, &\text{ if } (i,j)\in\{(0,1),(2,1)\},\\
i, &\text{ if } j=2,\\
(i+1)\,\mathrm{mod}\,4, &\text{ if } j=0,\\
c(v_{i,j-2}),& \text{ otherwise}.
\end{array}
\end{cases}\]
(Figure \ref{Fig_K4P3} illustrates the graph $K_4\boxdot P_3$.)

\begin{figure}[ht]
\begin{center}
\includegraphics[scale=0.125]{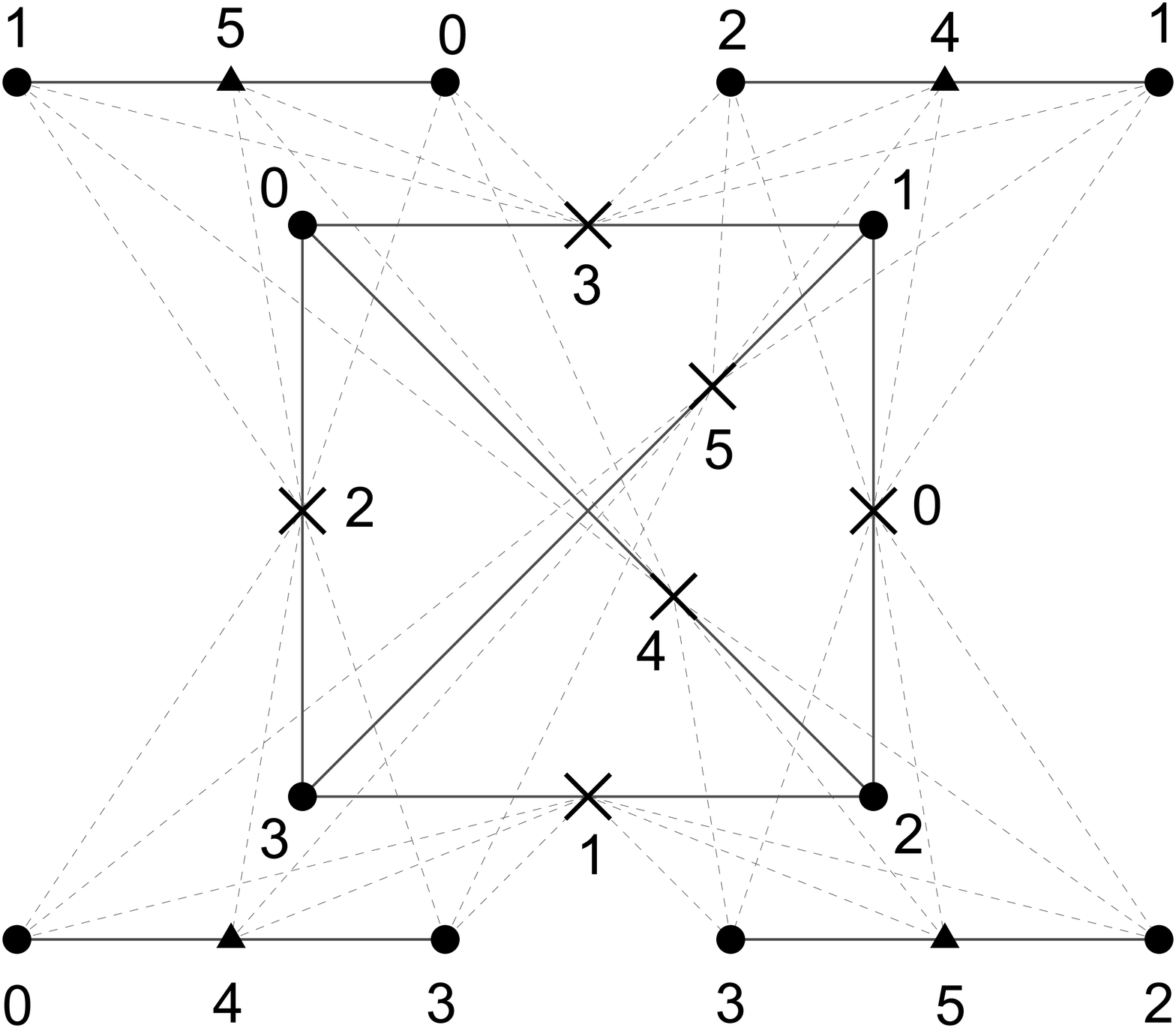}
\caption{Optimal $b$-chromatic coloring of $K_4\boxdot P_3$.}
\label{Fig_K4P3}
\end{center}
\end{figure}

Further, if $n\geq 7$ and $t=3$, then $c(s_{i,j})=(i+j+1)\,\mathrm{mod}\,(n+1)$, whenever $n$ is odd. Otherwise, if $n$ is even, then, for each positive integer $h\leq \frac n2$, we define
\[c(s_{i,i+h})=\begin{cases}
(i-2)\,\mathrm{mod}\,(n+1), & \text{ if } h=1,\\
\left(i+\frac {n-h}2\right)\,\mathrm{mod}\,(n+1), & \text{ if } 1<h<\frac n2 \text{ and } h \text{ is even },\\
\left(i-\frac{h-1}2\right)\,\mathrm{mod}\,(n+1), & \text{ if } 1<h<\frac n2 \text{ and } h \text { is odd },\\
i+1, & \text{ if } i<h = \frac n2.
\end{cases}\]
In addition, we define
\[c(v_{i,k})=\begin{cases}
\begin{array}{ll}
c(u_i),& \text{ if } k=1,\\
(2i+1)\,\mathrm{mod}\,(n+1),& \text{ if } n \text{ is odd and } k\in\{0,2\},\\
c(s_{i-1,i+1}),& \text{ otherwise}.
\end{array}
\end{cases}\]
According to this definition of the map $c$, we have that, if $n$ is odd, then a $b$-rainbow set is formed by the vertices $v_{0,1},\ldots,v_{n,1},\,v_{\frac{n-1}2,2}$. Otherwise, if $n$ is even, then a $b$-rainbow set is formed by the vertices $v_{0,1},v_{0,2},\ldots,v_{\frac n2-1,1},$ $v_{\frac n2-1,2},\,v_{\frac n2,1}$. (Figure \ref{Fig_K9P3} illustrates the graphs $K_8\boxdot P_3$ and $K_9\boxdot P_3$.)

\begin{figure}[ht]
\begin{center}
\includegraphics[scale=0.09]{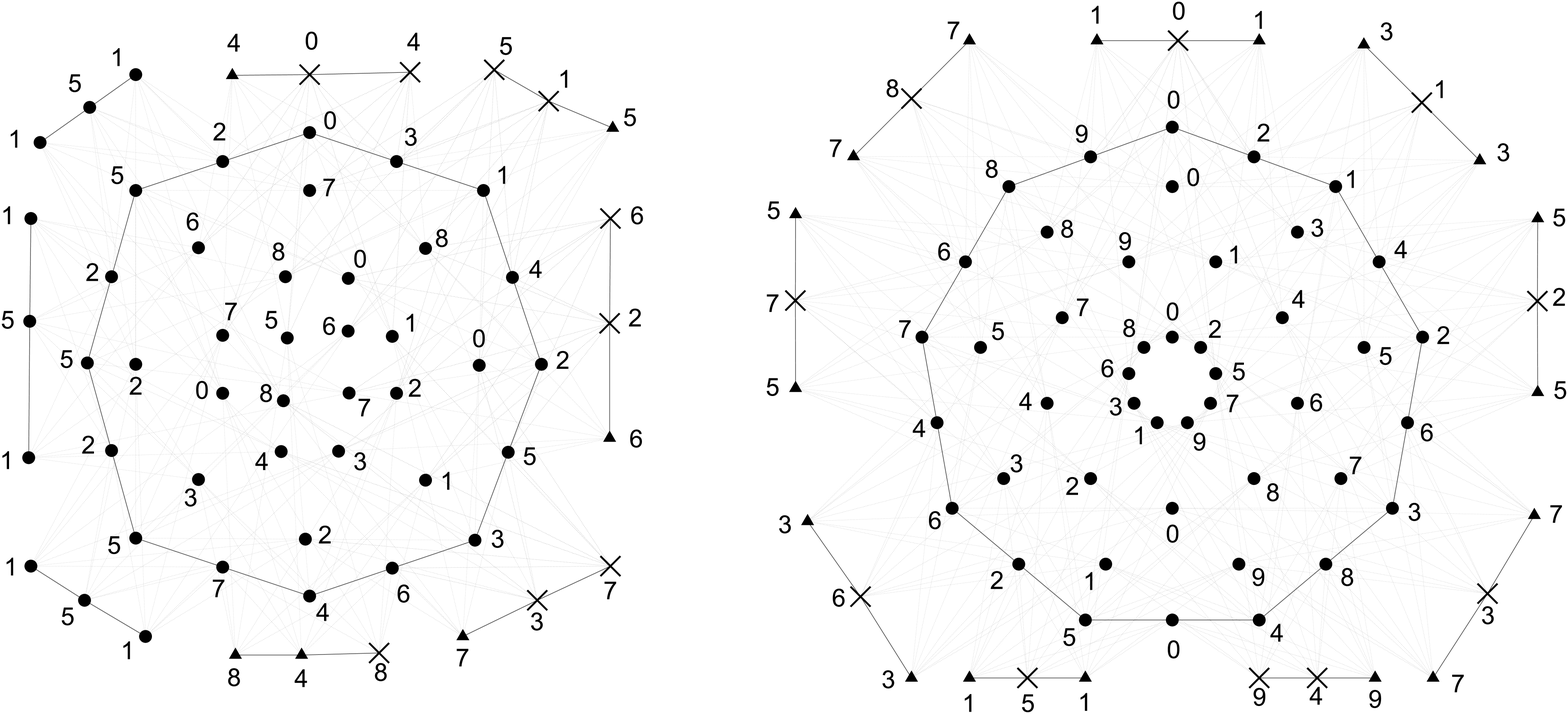}
\caption{Optimal $b$-chromatic colorings of $K_8\boxdot P_3$ and $K_9\boxdot P_3$.}
\label{Fig_K9P3}
\end{center}
\end{figure}

Finally, if $n\geq 2t+1>7$, then we define the map $c$ as in the previous case, except for
\[c(v_{i,k})=\begin{cases}
\begin{array}{ll}
n+1, & \text{ if } \begin{cases}
n \text{ is even and } \begin{cases}
i\neq \frac n2 \text{ and } k\in\{0,3\},\\
(i,k)=\left(\frac n2,1\right), 
\end{cases}\\
n \text{ is odd and }  \begin{cases}
i\neq \frac {n-1}2 \text{ and } k=2,\\
i=\frac {n-1}2 \text{ and } k\in\{0,3\},
\end{cases}
\end{cases}\\
c(v_{i,j-3}), & \text{ if } j\in\{4,5\},\\
c(v_{i,k-2}),& \text{ otherwise}.
\end{array}
\end{cases}\]
A $b$-rainbow set is formed by the same $b$-vertices of the previous case, together with $v_{\frac n2,1}$, if $n$ is even, and $v_{0,2}$, if $n$ is odd. (Figure \ref{Fig_K9P4} illustrates the graphs $K_8\boxdot P_4$ and $K_9\boxdot P_4$.)
\end{proof}

\begin{figure}[ht]
\begin{center}
\includegraphics[scale=0.08]{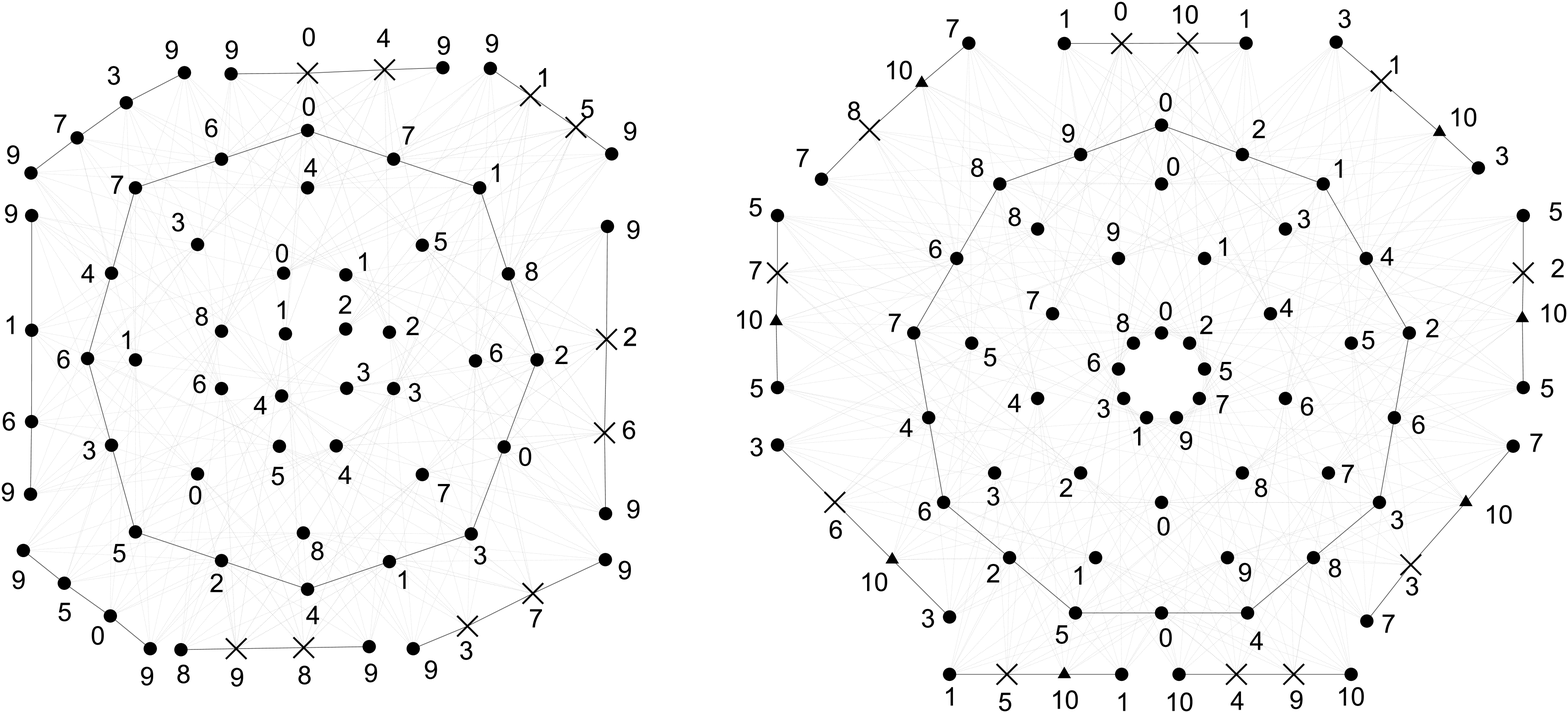}
\caption{Optimal $b$-chromatic colorings of $K_8\boxdot P_4$ and $K_9\boxdot P_4$.}
\label{Fig_K9P4}
\end{center}
\end{figure}

\vspace{0.2cm}

The next graph to study is the SVN corona $K_n\boxdot C_t$ such that $m(K_n\boxdot C_t)\leq n+2$. From Lemma \ref{lemma1} and (\ref{eq_d}), it is equivalent to say that either $n\in\{2,3,4\}$ or $n\geq 2t+1$.

\begin{theorem}\label{theorem_KnCt} Let $n>1$ and $t>2$ be two positive integers. Then,
\[\varphi(K_n \boxdot C_t)=
\begin{cases}
\begin{array}{ll}
n+1,& \text{ if }\begin{cases}
(n,t)=(2,4),\\
n\geq 9 \text{ and } t=4,
\end{cases}\\
n+2, & \text{ if } \begin{cases}
n=2 \text{ and } t\neq 4,\\
n\in\{3,4\},\\
n\geq 2t+1 \text{ and } t\neq 4.
\end{cases}
\end{array}
\end{cases}\]
\end{theorem}

\begin{proof} The case $n=3$ follows from Theorem \ref{theorem_CnCt}. In addition, since $m(K_2\boxdot C_t)=4$, the case $n=2$ and $t\neq 4$ follows from Lemmas \ref{lemma1} and \ref{lemma_Kn_n2}, together with Proposition \ref{proposition_Kouider}. Moreover, it is readily verified the non-existence of a $b$-rainbow set of $K_2\boxdot C_4$ formed by four distinct $b$-vertices. Thus, the same mentioned results imply that $\varphi(K_2\boxdot C_4)=3$. 

Let us focus now on the case $n\geq 2t+1$, for which  Lemma \ref{lemma1} implies that $\varphi(K_n\boxdot C_4)\leq m(K_n\boxdot C_4)=n+2$. In order to prove that this upper bound is not reached, let us suppose the existence of an $(n+2)$-coloring of $K_n\boxdot C_4$. If $v_{i_0,j_0}$ were a vertex of a $b$-rainbow set of $K_n\boxdot C_4$, for some $i_0<n$ and $j_0<4$, 
then the four vertices $v_{i_0,0},v_{i_0,1},v_{i_0,2}$ and $v_{i_0,3}$ would be colored by at most three colors. One of them would be the color $c(u_{i_0})$, which makes that no vertex of the form $s_{i_0,k}$ may be part of the $b$-rainbow set under consideration. Then, since $c$ is a proper coloring, it would be $c(v_{i_0,k})=c(v_{i_0,(k+2)\,\mathrm{mod}\,4})$, for some $k<4$, and hence, from the mentioned four vertices, only the vertex $v_{i_0,j}$ would be part of the $b$-rainbow set. It contradicts the case $n>9$, for which only vertices of the form $v_{i,j}$ may be part of the $b$-rainbow set and hence, the latter only could be formed by at most $n$ vertices. Based on the previous remarks, a simple study of cases enables us to ensure that this condition is also no feasible in case of being $n=9$. Hence, $\varphi(K_n\boxdot C_4)\leq n+1$, for all $n\geq 9$. In order to prove that this upper bound is reached, it is enough to consider the $b$-chromatic coloring $c$ of $K_n\boxdot C_4$ described in the proof of Theorem \ref{theorem_KnPt}, except for

\[c(v_{i,j})=\begin{cases}
\begin{array}{ll}
c(u_{\frac n2}),& \text{ if } n \text{ is even}, i=\frac n2 \text{ and } j\in\{1,3\},\\
c(v_{i,1}), & \text{ if } n\text{ is even}, i\neq \frac n2 \text{ and } j=3,\\
c(v_{i,2}), & \text{ if } n\text{ is even and } j=0,\\
n+1, & \text{ if } n \text{ is odd},  i=\frac {n-1}2 \text{ and } j=0,\\
c(v_{i,0}), & \text{ if } n\text{ is odd}, i\neq \frac {n-1}2 \text{ and } j=2,\\
c(v_{i,1}), & \text{ if } n\text{ is odd and } j=3.
\end{array}
\end{cases}\]

Finally, Lemma \ref{lemma1} also implies that the remaining values described in the statement of this theorem are upper bounds of $\varphi(K_n\boxdot C_t)$. The same $b$-chromatic coloring described for both $n=4$ and $n\geq 2t+1\geq 11$ in the proof of Theorem \ref{theorem_KnPt} enables us to ensure that these upper bounds are reached in such cases. Here, we assume that $C_t=\langle\,v_0,\ldots,v_{t-1},v_0\,\rangle$. For $n\geq 2t+1=7$, it is also enough to consider the same $b$-chromatic coloring $c$ described in the proof of Theorem \ref{theorem_KnPt}, together with $c(v_{i,2})=n+1$, for every positive integer $i<n$.
\end{proof}

\vspace{0.2cm}

Now, in order to study the SVN corona of a complete graph and either a star or a complete graph, the following technical result is useful.

\begin{lemma}\label{lemma_technical} Let $n>1$ be a positive integer and let $H$ be a graph of order $t\geq 1$ such that $\varphi(K_n\boxdot H)\geq 2t+2$. If $\mathcal{R}$ is a $b$-rainbow set of $K_n\boxdot H$ and $\langle\,u_{i_0},u_{i_1},\ldots,u_{i_{\ell}},u_{i_0}\,\rangle$ is a cycle in $K_n$, with $\ell<n$, such that $\{s_{i_0,i_1},s_{i_1,i_2},\ldots,s_{i_{\ell-1},i_{\ell}},s_{i_{\ell},i_0}\}\subseteq\mathcal{R}$, then $\ell$ must be even. Moreover, if $n\leq 2t+1$, then $\ell=2t=n-1$.
\end{lemma}

\begin{proof} Without loss of generality, let us suppose the existence of an optimal $b$-chromatic coloring $c$ of the graph $K_n\boxdot H$, for which there are a $b$-rainbow set $\mathcal{R}$ and a cycle $\langle\,u_0,u_1,\ldots,u_\ell,u_0\,\rangle$ in $K_n$, with $\ell<n$, such that $\{s_{0,1},s_{1,2},\ldots,s_{\ell-1,\ell},s_{0,\ell}\}\subseteq\mathcal{R}$. 

Since $s_{0,1}$ and $s_{1,2}$ are two distinct b-vertices of the b-rainbow set $\mathcal{R}$, it must be $c\left(s_{0,1})\in c(N(s_{1,2})\setminus N(s_{0,1})\right)$. That is, $c(s_{0,1})\in\{c(u_2),\,c(v_{2,0}),\ldots,c(v_{2,t-1})\}$. Moreover, we have in a similar and recursive way that $c(s_{0,1})\in c(N(s_{2\kappa-1,2\kappa})\setminus$ $N(s_{2\kappa-2,2\kappa-1}))$, and hence,  $c(s_{0,1})\in\{c(u_{2\kappa}),\,c(v_{2\kappa,0}),\ldots,c(v_{2\kappa,t-1})\}$, for every positive integer $\kappa\leq \frac \ell 2$. If $\ell$ is odd, then $c(s_{0,1})\in c(N(s_{0,1}))$. It contradicts the fact that $c$ is a proper coloring. So, $\ell$ must be even.

Furthermore, for each non-negative integer $i<n$, $\left\{\left\{c(u_i),\,c(v_{i,0}),\ldots,c(v_{i,t})\right\}\right\}\setminus\left\{c(s_{i+1,i+2})\right\} = \left\{\left\{c(u_{i+2}),\,c(v_{i+2,0}),\ldots,c(v_{i+2,t})\right\}\right\}\setminus\left\{c(s_{i,i+1})\right\}$, where all the indices are taken modulo $n$. Since $\ell$ is even, then the sets
\[\left\{\left\{c(u_0),\,c(v_{0,0}),\ldots,c(v_{0,t-1})\right\}\right\}\setminus\left\{c(s_{2\kappa+1,2\kappa+2})\colon\, 0\leq \kappa < \frac \ell 2\right\}\]
and
\[\left\{\left\{c(u_\ell ),\,c(v_{\ell,0}),\ldots,c(v_{\ell,t-1})\right\}\right\}\setminus\left\{c(s_{2\kappa,2\kappa+1})\colon\, 0\leq \kappa< \frac \ell 2\right\}\]
coincide. These two sets are indeed formed by at most one common color, because, since $s_{0,\ell}$ is a $b$-vertex, it must be $2t+1\leq |c(N(s_{0,\ell}))|\leq d(s_{0,\ell})=2t+2$. That is, $t+1-\frac \ell 2\leq 1$. Therefore, if $n\leq 2t+1$, then we have that $\ell<n\leq 2t+1$, and thus, $t+1-\frac \ell 2 >\frac 12$. Hence, $t+1-\frac \ell 2=1$. That is, $\ell=2t=n-1$.
\end{proof}

\vspace{0.2cm}

Let us study the SVN corona $K_n\boxdot S_t$ such that $m(K_n\boxdot S_t)\leq n+2$. From Lemma \ref{lemma1} and (\ref{eq_d}), it is equivalent to say that $n\geq 2t+3$.

\begin{theorem}\label{theorem_KnSt} Let $n>1$ and $t>2$ be two positive integers. Then,
\[\varphi(K_n \boxdot S_t)=
\begin{cases}
\begin{array}{ll}
n,& \text{ if } n=2t+3,\\
n+1, & \text{ if } n\geq 2t+4.
\end{array}
\end{cases}\]
\end{theorem}

\begin{proof} From Lemma \ref{lemma1} and (\ref{eq_d}), we have that $\varphi(K_n\boxdot S_t)\leq m(K_n\boxdot S_t)=n+1$, whenever $n\geq 2t+4$, and $\varphi(K_{2t+3}\boxdot S_t)\leq m(K_{2t+3}\boxdot S_t)=n+2$. In order to prove that this upper bound is reached whenever $n\geq 2t+4$, it is enough to define the same $b$-chromatic coloring $c$ described in the proof of Theorem \ref{theorem_KnPt} for the graph $K_n\boxdot P_3$, for $n\geq 7$, together with $c(v_{i,j})=c(v_{i,2})$, for every pair of non-negative integers $i<n$ and $j\in\{3,\ldots,t\}$. Here, we assume that $V(S_t)=\{v_0,\ldots,v_t\}$, where $v_0$ is the center of the star.

Now, let us suppose the existence of an optimal $b$-chromatic coloring $c$ of $K_{2t+3}\boxdot S_t$, with $k\in\{2t+4,\,2t+5\}$ colors, and let $\mathcal{R}$ be a $b$-rainbow set arising from this coloring. From Lemma \ref{lemma_technical}, any subset $\{s_{i_0,i_1},\,s_{i_1,i_2},\ldots,s_{i_{\ell-1},i_\ell},\,s_{i_\ell,i_0}\}\subseteq \mathcal{R}$ arising from a cycle within $K_n$ would be such that $\ell=2t+3$. However, it is readily verified that $\mathcal{R}$ can only contain non-adjacent vertices of the form $s_{i,j}$ or $v_{i,0}$, with $0\leq i,j<2t+3$. So, it would be exactly formed by the $2t+3$ vertices of the complete graph, which contradicts that $k\in\{2t+4,2t+5\}$. Hence, no such cycle can exist. But then, the mentioned non-adjacency of $b$-vertices makes $\mathcal{R}$ to be formed by, at most, $n$ distinct $b$-vertices. It contradicts again that $k\in\{2t+4,2t+5\}$. Therefore, $\varphi(K_{2t+3}\boxdot S_t)\leq 2t+3$. In order to prove that this upper bound is reached, it is enough to consider the $b$-chromatic coloring $c$ of the graph $K_{2t+3}\boxdot S_t$ such that, for each non-negative integer $i<2t+3$ and each pair of positive integers $h\leq t$ and $k\leq t$, we have that $c(u_i)=c(v_{i,0})=i\,\mathrm{mod}\,(2t+3)$, $c(s_{i,i+h})=\left(i+\left(2\cdot\left\lfloor\frac h2\right\rfloor-1\right)\right)\,\mathrm{mod}\,(2t+3)$ and $c(v_{i,k})=(i+2k)\,\mathrm{mod}\,(2t+3)$. A $b$-rainbow set is formed by the vertices $s_{0,1},\,s_{1,2},\,\ldots,s_{n-2,n-1},\,s_{0,n-1}$. (Figure \ref{Fig_K9S3} illustrates the graph $K_9\boxdot S_3$.)
\end{proof}

\begin{figure}[ht]
\begin{center}
\includegraphics[scale=0.095]{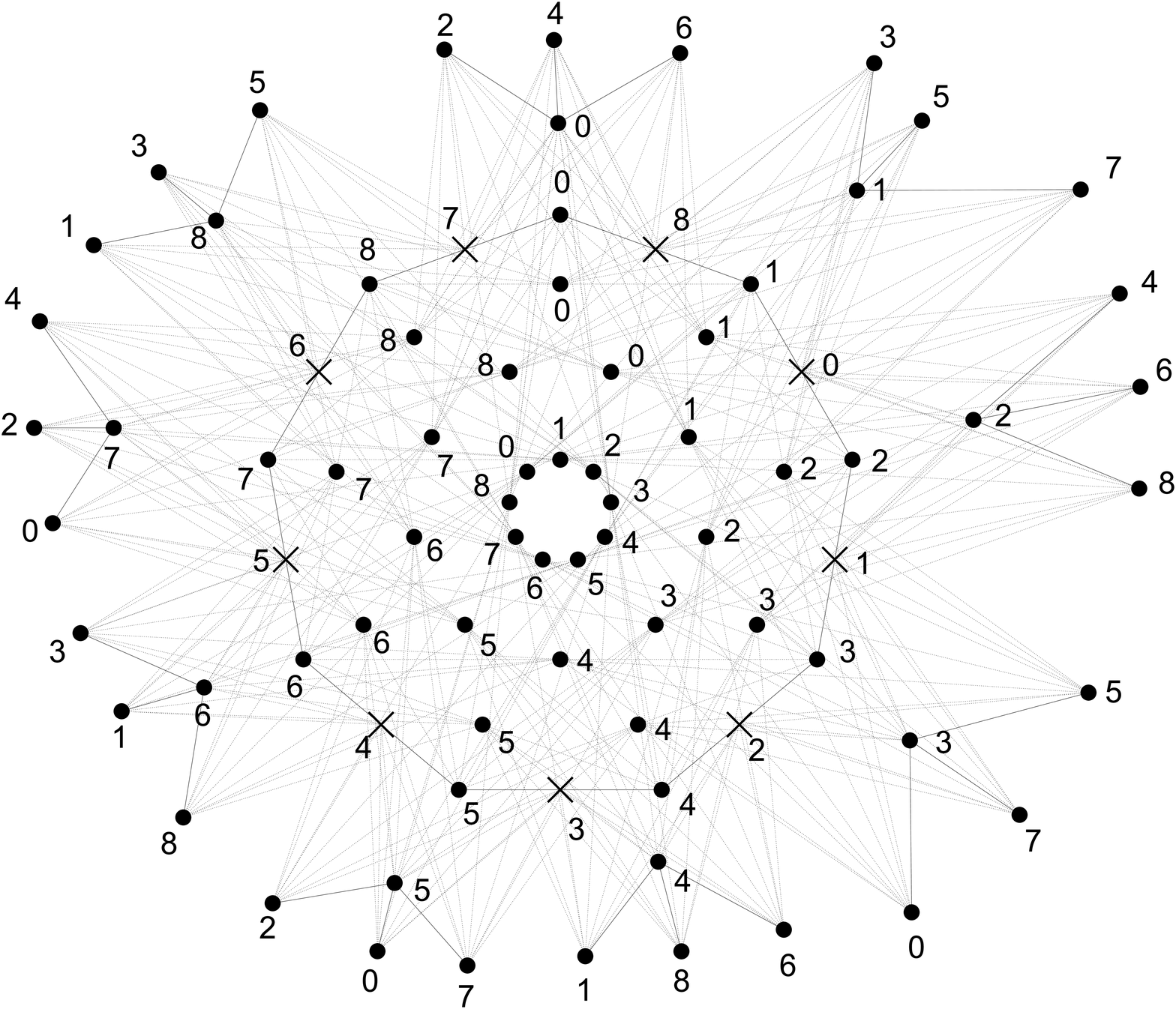}
\caption{Optimal $b$-chromatic coloring of $K_9\boxdot S_3$.}
\label{Fig_K9S3}
\end{center}
\end{figure}

Let us finish our study by dealing with the SVN corona $K_n\boxdot K_t$ in case of being $m(K_n\boxdot K_t)\leq n+2$. From Lemma \ref{lemma1} and (\ref{eq_d}), it is equivalent to say that $t\in\{1,2,3\}$, except for $(n,t)\in\{(5,3),\,(6,3)\}$.

\begin{theorem}\label{theorem_KnKt} Let $n>1$ and $t$ be two positive integers. Then,
\[\varphi(K_n \boxdot K_t)=
\begin{cases}
\begin{array}{ll}
n-1, & \text{ if } t=1 \text{ and } n>4 \text{ is even},\\
n, & \text{ if } \begin{cases}
t=1 \text{ and } n \text{ is odd},\\
(n,t)\in\{(2,1),\,(4,1),\,(5,2)\},
\end{cases}\\
n+1, & \text{ if }\begin{cases}
(n,t)\in\{(2,2),\,(3,2)\},\\
n\geq 6 \text{ and } t=2,
\end{cases}\\
n+2, & \text{ if } \begin{cases}
(n,t)\in\{(2,3),\,(3,3),\,(4,2),\,(4,3)\},\\
n\geq 7 \text{ and } t=3.
\end{cases}
\end{array}
\end{cases}\]
\end{theorem}

\begin{proof} The case $t=3$ holds from Theorem \ref{theorem_KnCt}. In addition, since $K_2\boxdot K_1$ coincides with $S_4$, the case $(n,t)=(2,1)$ follows from Proposition \ref{proposition_Kouider}. Moreover, even if  $m(K_4\boxdot K_1)=5$ and $m(K_5\boxdot K_2)=7$, it follows readily from Lemma \ref{lemma_technical} that $\varphi(K_4\boxdot K_1)\leq 4$ and $\varphi(K_5\boxdot K_2)\leq 5$. The second and fifth graphs in Figure \ref{Fig_K2K1} illustrate that both upper bounds are reached.

\begin{figure}[ht]
\begin{center}
\includegraphics[scale=0.06]{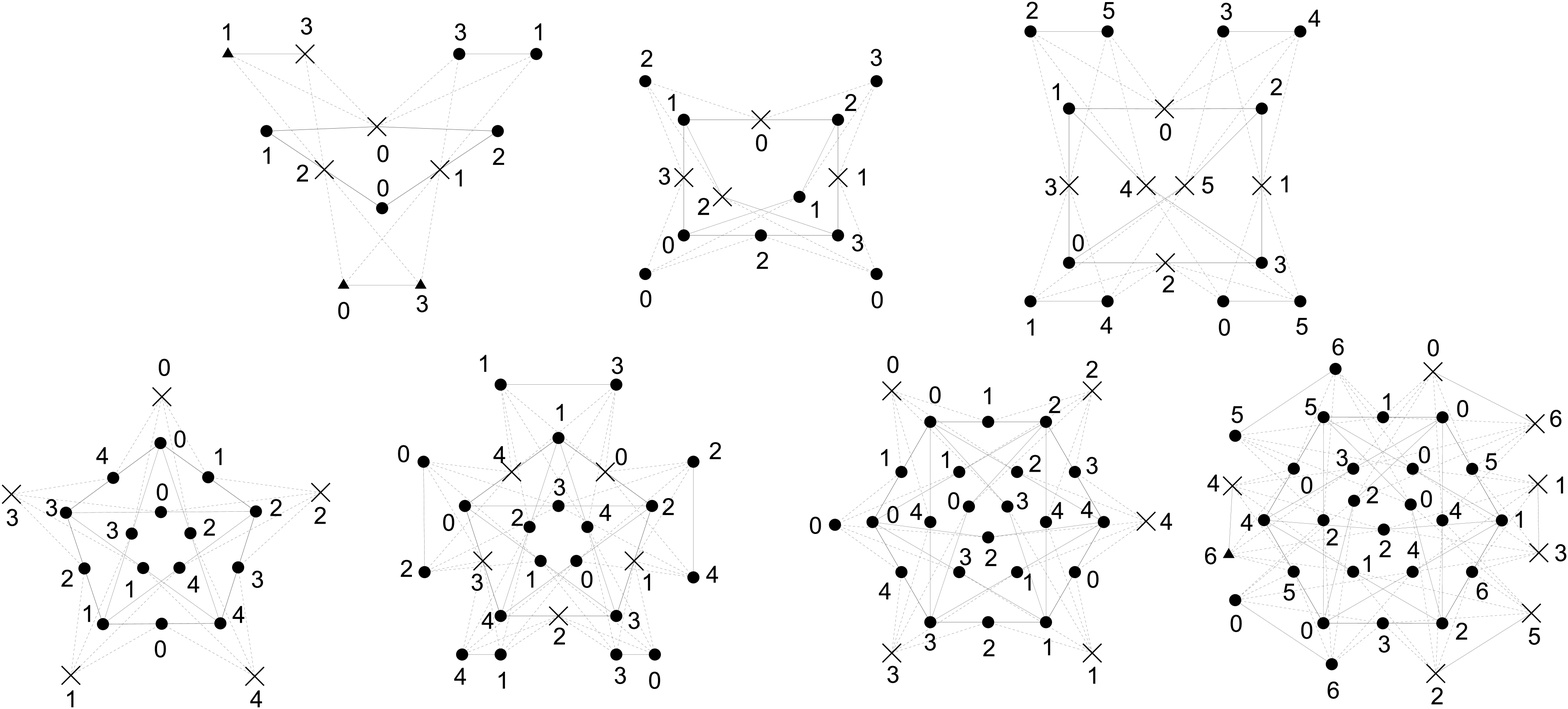}
\caption{Optimal $b$-chromatic coloring of $K_n\boxdot K_t$, for all $(n,t)\in\{(3,2),\,(4,1),\,(4,2),\,(5,1),\,(5,2),\,(6,1),\,(6,2)\}$.}
\label{Fig_K2K1}
\end{center}
\end{figure}

Further, even if $m(K_n\boxdot K_1)=n$, for all $n$, the only $b$-rainbow sets of a $b$-chromatic coloring $c$ of $K_n\boxdot K_1$, with an even number $n>4$ of colors, would be those ones containing one vertex of each one of the $n$ sets $\{u_i,\,v_{i,0}\}$, with $0\leq i<n$. But then, all the $n$ colors should appear the same number of times in the multiset $\{c(s_{i,j})\colon\, 0\leq i,j<n\}$, which is not possible because $\binom {n}2$ is not a multiple of the even integer $n$. Hence, $\varphi(K_n\boxdot K_1)\leq n-1$, for every even integer $n>4$. This upper bound is reached, because of the $b$-chromatic proper coloring $c$ such that, for every pair of non-negative integers $i,j<n$, with $i<j$, we have that $c(u_i)=c(v_{0,i})=2i\,\mathrm{mod}\,(n-1)$  and $c(s_{i,j})=(i+j)\,\mathrm{mod}\,(n-1)$, except for $c(s_{0,n-1})=1$. Here,  $V(K_t)=\{v_0,\ldots,v_{t-1}\}$.  Then, a $b$-rainbow set is formed by the vertices $v_{0,0},\,\ldots,v_{n-1,0}$. (The sixth graph in Figure \ref{Fig_K2K1} illustrates the case $K_6\boxdot K_1$.)

From Lemma \ref{lemma1} and (\ref{eq_d}), all the remaining values are upper bounds of the $b$-chromatic number. Then, the case $n=t=2$ follows from Proposition \ref{proposition_Kouider} and Lemma \ref{lemma_Kn_n2}. Moreover, the case $(n,t)\in\{(3,2),\,(4,2),\,(6,2)\}$ are illustrated by the first, third and seventh graphs in Figure \ref{Fig_K2K1}. In order to see that the remaining upper bounds are reached, we define an appropriate $b$-chromatic coloring $c$ of the graph $K_n\boxdot K_t$. The same $b$-chromatic coloring defined for the graph $K_6\boxdot K_1$ in Figure \ref{Fig_K2K1}, together with $c(v_{i,1})=5$, for every non-negative integer $i<6$, constitutes a $b$-chromatic coloring for the graph $K_6\boxdot K_2$. In addition, if $n\geq 7$ and $t=2$, then it is enough to define $c$ as the restriction to $K_n\boxdot K_2$ of the $b$-chromatic coloring  for $K_n\boxdot K_3$ that was described in the proof of Theorem \ref{theorem_KnPt}. Finally, if $t=1$ and $n$ is odd, then, for each triple of positive integers $i,j<n$ and $k<t$, we define $c(u_i)=c(v_{i,0})=2i\,\mathrm{mod}\,n$ and $c(s_{i,j})=(i+j)\,\mathrm{mod}\,n$. A $b$-rainbow set is formed by the vertices $v_{0,0},\,\ldots,v_{n-1,0}$. (The fourth graph in Figure \ref{Fig_K2K1} illustrates the case $n=5$.)
\end{proof}

\vspace{0.2cm}

The same $b$-chromatic coloring described in the proof of Theorem \ref{theorem_KnKt} for the graph $K_n\boxdot K_1$, with $n$ odd, gives rise to a $b$-chromatic coloring for the subdivision graph $S(K_n)$. Since $m(S(K_n))=n$, we have from Lemma \ref{lemma1} that $\varphi(S(K_n))=n$, whenever $n$ is odd. Note here that Vijayalakshmi \cite[Theorem 3.1] {Vijayalakshmi2014} (see also \cite[Theorem 2.3]{Jeeva2017}) indicated without proof that the central graph of the complete graph $K_n$, with $n>3$, has $b$-chromatic number equal to $n-1$. Since this central graph is isomorphic to the subdivision graph $S(K_n)$, their claim is false.

\section{Conclusion and further work}\label{sec:conclusions}

As a first approach to deal with the $b$-chromatic coloring of SVN coronas, we have determine the $b$-chromatic number of the SVN coronas $G\boxdot H$ and $G\boxdot K_n$, with each $G$ and $H$ being either a path, or a cycle, or a star, and $K_n$ being the complete graph of order $n$. The case $K_n\boxdot G$ has also been solved in those cases in which $m(K_n\boxdot G)\leq n+2$. A significant number of technical results based on study of cases, together with their constructive proofs and illustrative examples, have been described to this end. As such, this paper may be considered as a starting point to delve into this topic.

The SVN corona $K_n\boxdot G$, where $m(K_n\boxdot G)> n+2$, seems to require a more extensive study of cases. We leave this for future work. Furthermore, the natural continuation of this paper is the study of the $b$-chromatic number of SVN coronas of other families of graphs, together with a similar approach to deal with the $b$-chromatic coloring of the so-called {\em subdivision-edge neighbourhood corona of graphs} \cite{Liu2013}.

\section*{Acknowledgements} Falc\'on's work is partially supported by the research project FQM-016 from Junta de Andaluc\'\i a.

\section*{Conflict of Interest}

The authors declare no conflict of interest.

\end{document}